\documentclass[12pt,leqno]{amsart}
\usepackage{amsmath, amssymb, amscd, amsfonts, xy,  stmaryrd, turnstile}

\usepackage
{hyperref}

\newtheorem{theorem}{Theorem}[section]
\newtheorem{lemma}[theorem]{Lemma}
\newtheorem{proposition}[theorem]{Proposition}
\newtheorem{corollary}[theorem]{Corollary}

\theoremstyle{definition}
\newtheorem{definition}[theorem]{Definition}
\newtheorem{example}[theorem]{Example}
\newtheorem{remark}[theorem]{Remark}

\begin{document}

\title[ FIP and FCP products]{FIP and FCP products of ring morphisms}

\author[G. Picavet and M. Picavet]{Gabriel Picavet and Martine Picavet-L'Hermitte}
\address{Universit\'e Blaise Pascal \\
Laboratoire de Math\'ematiques\\ UMR6620 CNRS \\ Les C\'ezeaux \\ 24, avenue des Landais\\
BP 80026 \\ 63177 Aubi\`ere CEDEX \\ France}

\email{Gabriel.Picavet@math.univ-bpclermont.fr, picavet.gm(at)wanadoo.fr}
\email{Martine.Picavet@math.univ-bpclermont.fr}

\begin{abstract} We   characterize some types of FIP and FCP ring extensions $R \subset S$, where $S$ is not an integral domain and $R$ may not be an integral domain, contrary to a general trend.  In most of the sections, $S$ is a product  of finitely many rings that are related to $R$ in various ways. A section  is devoted to the case where $S$ is the idealization of  an $R$-module. As a by-product  we exhibit characterizations of the modules that have finitely many submodules. Ring extensions of the form $R^n \hookrightarrow R^p$ associated to some matrices are also considered. The paper ends with the FIP property of \'etale morphisms. Our tools are minimal ring morphisms and seminormalization, while Artinian conditions on rings are ubiquitous. 
\end{abstract}

\subjclass[2010]{Primary:13B02, 13B21, 13B22; Secondary: 13E10, 13B30, 13B40, 05A18, 13F10}

\keywords {$\Delta_0$-extension, $\Delta$-extension, quadratic extension, SPIR, Idealization, Stirling number,  crucial ideal, decomposed, inert, ramified extension, flat epimorphism, FIP, FCP extension, minimal ring extension, infra-integral extension, integral extension, integrally closed, Artinian ring, seminormal,  subintegral, support of a module}

\maketitle

\section{Introduction and Notation}

 All rings $R$ considered  are commutative, nonzero and unital; all morphisms of rings are unital. Let $R\subseteq S$ be a (ring) extension. The set of all $R$-subalgebras of $S$  is denoted by $[R,S]$. The extension $R\subseteq S$ is said to have FIP (for the ``finitely many intermediate algebras property") if $[R,S]$ is finite. A {\it chain} of $R$-subalgebras of $S$ is a set of elements of $[R,S]$ that are pairwise comparable with respect to inclusion. We say that the extension $R\subseteq S$ has FCP (for the ``finite chain property") if each chain of $R$-subalgebras of $S$ is finite. It is clear that each extension that satisfies FIP must also satisfy FCP. Our main tool are the minimal (ring) extensions, a concept  introduced by Ferrand-Olivier \cite{FO}. Recall that an extension $R\subset S$ is called {\it minimal} if $[R,S]=\{R,S\}$. The key connection between the above ideas is that if $R\subseteq S$ has FCP, then any maximal (necessarily finite) chain $R=R_0\subset R_1\subset\cdots\subset R_{n-1}\subset R_n=S$, of $R$-subalgebras of $S$, with {\it length} $n<\infty$, results from juxtaposing $n$ minimal extensions $R_i\subset R_{i+1},\ 0\leq i\leq n-1$.  Following \cite{J}, the {\it length of} $[R,S]$, denoted by $\ell[R,S]$, is the supremum of the lengths of chains of $R$-subalgebras of $S$. In particular, if $\ell[R,S]=r$, for some integer $r$, there exists a maximal chain $R=R_0\subset R_1\subset\cdots\subset R_{r-1}\subset R_r=S$ of $R$-subalgebras of $S$ with length $r$. Against the general trend, we characterized arbitrary FCP and FIP extensions  in \cite{DPP2}, a joint paper by D. E. Dobbs and ourselves whereas most of papers on the subject are concerned with extensions of integral domains. It is worth noticing here that FCP extensions of integral domains are generally nothing but  extensions of overrings as a quick look at  \cite[Theorems 4.1,4.4]{CDL2} shows because FCP extensions are composite of minimal extensions.
 
In  this paper, we take the opposite way  and consider the FCP or FIP properties for special types of extensions, like $K \to K^n$ where $K$ is a field. It is known that these extensions have FIP and actually this example  motivated us to study  generalizations.   These extensions are integral and most of time seminormal within  the meaning of Swan. Problems arise when they are not seminormal, leading to the computation of seminormalizations.
The seminal  work on FIP and  FCP by R. Gilmer is settled for $R$-subalgebras of $K$ (also called overrings of $R$), where $R$ is a domain and $K$ its quotient field. In particular, \cite[Theorem 2.14]{G} shows that $R \subseteq S$ has FCP for each overring $S$ of $R$ only  if $R/C$ is an Artinian ring, where $C = (R: \overline{R})$ is the conductor of $R$ in its integral closure. This necessary Artinian  condition is  not surprisingly present in all our results.   

We now give a slight outline of our work, slight because results are too technical to be discussed in an introduction. Roughly speaking, we are concerned by product morphisms $R \to \prod_{i=1}^n R_i$ that are extensions. We will observe that  results  may depend on the value of  $n$.

In Section 2, we look at  diagonal extensions  $R\subseteq\prod_{i=1}^nR_i$, for some finitely many FCP or FIP extensions $R\subseteq R_i$. When $R\subseteq R_i$ has FCP for each $i$,  Proposition~\ref{2.11} asserts  that  $R\subseteq\prod_{i=1}^nR_i$ has FCP if and only if $R$ is an Artinian ring. The FIP condition is much more complicated. For instance,  $R$ has finitely many ideals if $R\subseteq\prod_ {i=1}^nR_i$ has FIP (Proposition~\ref{2.2}).  Moreover, $R\subseteq R^2$ has FIP if and only if $R$    has finitely many ideals (Corollary~\ref{2.5}). 

 Section 3 is concerned with extensions of the form $R/\cap_{j=1}^nI_j\subseteq\prod_{j=1}^n(R/I_j)$, where $I_1,\ldots,I_n$ are proper ideals of a ring $R$, not necessarily distinct. In particular, for $n=2$, we get a generalization of the Chinese Remainder Theorem. 
 
 Section 4 is devoted to diagonal extensions $R \subseteq R^n$. We  get in (Theorem~\ref{4.2}) that $R\subseteq R^n$ has FIP if and only if $R$ has finitely many ideals and $n\leq 2$ as soon as there exists a maximal ideal $M$ of $R$ such that $R_M$ is not a field and $R/M$ is an infinite field. We show that $ R^n$ may have different structures of $R^p$-algebras if $p<n$ are two positive integers, leading to different occurrences  of FIP extensions $R^p\hookrightarrow R^n$. 

 Section 5 is concerned with   $R$-modules $M$  over a ring $R$ and  ring extensions $R\subseteq R(+)M$, where $R(+)M$ is the idealization of $M$. The main results are as follows.  (Proposition~\ref{5.2}) shows that $R\subseteq R(+)M$ has FCP if and only if the length of the $R$-module $M$ is finite, while (Proposition~\ref{5.4}) says that $R\subseteq R(+)M$ has FIP if and only if $M$ has finitely many $R$-submodules. This leads us to characterize $R$-modules having finitely many $R$-submodules in Corollary~\ref{5.6}. An $R$-module $M$, with $C:= (0:M)$, has finitely many submodules if and only if the three following conditions are satisfied: $M$ is finitely generated, $R/C$ has finitely many ideals and $M_P$  is cyclic for any prime ideal $P$ of $R$ containing $C$ such that $R/P$ is infinite. Then Theorem~\ref{5.12} gives a structure theorem for these modules that are faithful.
 
Etale morphisms are considered in Section 6, because separable algebraic extensions of fields are known to have FIP. In order to extend this result, we have to add a seminormality condition that is trivial for algebraic extensions.

Let $R$ be a ring. As usual, Spec$(R)$ (resp$.$ Max$(R)$) denotes the set of all prime ideals (resp$.$ maximal ideals) of $R$. If $I$ is an ideal of $R$, we set ${\mathrm V}_R(I):=\{P\in\mathrm{Spec}(R)\mid I\subseteq P\}$. If $R\subseteq S$ is a ring extension and $P\in\mathrm{Spec}(R)$, then $S_P$ is the localization $S_{R\setminus P}$  and $(R:S)$ is the conductor of $R\subseteq S$. When there is no possible confusion, we denote the integral closure of $R$ in $S$ by $\overline R$. Recall that if $E$ is an $R$-module, its {\it support} $\mathrm{Supp}_R(E)$ is the set of prime ideals $P$ of $R$ such that $E_P\neq 0$ and $\mathrm{MSupp}_R(E):=\mathrm{Supp}_R(E)\cap\mathrm{Max}(R)$. If $E$ is an $R$-module, ${\mathrm L}_R(E)$ is its length. We will shorten finitely generated module into f.g$.$ module. Recall that a {\it special principal ideal ring} (SPIR) is a principal ideal ring $R$ with a unique nonzero prime ideal $M=Rt$, such that $M$ is nilpotent of index $p>0$. Hence a SPIR is not a field. Each nonzero element of a SPIR is of the form $ut^k$ for some unit $u$ and some {\it unique} integer $k< p$.  Finally, as usual, $\subset$ denotes proper inclusion and $|X|$ denotes the cardinality of a set $X$.

We sum up the fundamental results on minimal  extensions we need.

\begin{theorem}\label{1.1} \cite{Dec}, \cite[Theorem 4.1]{DPPS}, \cite[Th\'eor\`eme 2.2 and Lemme 3.2]{FO} and \cite[Proposition 3.2]{Pic}. Let $A\subset B$ be a minimal extension with associated inclusion map $f: A \to B$. Then:\\
\indent (a) There is some $M\in\mathrm{Max}(A)$, called the {\bf crucial (maximal) ideal} of $A\subset B$, such that $A_P=B_P$ for each $P\in\mathrm{Spec}(A)\setminus\{M\}$. We denote this ideal by $\mathcal{C}(A,B)$.\\
\indent (b) The following three conditions are equivalent:\\
\indent \indent (1) There is a prime ideal in $B$ lying over $M$;\\
\indent \indent (2) $MB=M$;\\
\indent \indent (3) $f$ is $($module-$)$finite.\\
\indent (c) The conditions of (b) do not hold if and only if $f$ is a flat epimorphism and then $(A:B)$ is a common prime ideal of $A$ and $B$ that is contained in $M$.\\
\indent (d)  There is a bijection $\mathrm{Spec}(B)\setminus{\mathrm V}_B(MB)\to\mathrm{Spec}(A)\setminus\{M\}$, with ${\mathrm V}_B(MB)=\emptyset$ when $f$ is a flat epimorphism. Moreover, for each $Q\in\mathrm{Max}(B)$, either $Q\cap A\in\mathrm{Max}(A)$ or $Q=(A:B)$.
\end{theorem}

Three types of minimal integral extensions exist, given by the next Theorem.

\begin{theorem}\label{1.2} \cite[Theorem 3.3]{Pic} Let $R\subset T$ be a ring extension and $M: =(R:T)$. Then $R\subset T$ is minimal and finite if and only if $M\in\mathrm{Max}(R)$ and one of the following three conditions holds:\\
\indent (a) {\bf inert case}: $M\in\mathrm{Max}(T)$ and $R/M\to T/M$ is a minimal field extension;\\
\indent (b) {\bf decomposed case}: There exist $M_1,M_2\in\mathrm{Max}(T)$ such that $M= M_1 \cap M_2$ and the natural maps $R/M\to T/M_1$ and $R/M\to T/M_2$ are both isomorphisms;\\
\indent (c) {\bf ramified case}: There exists $M'\in\mathrm{Max}(T)$ such that ${M'}^ 2 \subseteq M\subset M',\  [T/M:R/M]=2$ (resp$.$ $\mathrm{L}_R(M'/M)= 1$), and the natural map $R/M\to T/M'$ is an isomorphism.\\
\indent In each of the above three cases, $M$ is the crucial ideal of $R \subset T$.
\end{theorem}

We also need some results about seminormality and t-closedness that we recall here. 

\begin{definition}\label{1.3} An integral extension  $f: R\hookrightarrow S$  is termed:

(1)  {\it infra-integral}  if all its residual extensions  are isomorphisms \cite{Pic 2}.  

(2)  {\it subintegral} if $f$ is infra-integral and  ${}^af$ is  bijective \cite{S}. 

\end{definition}

A minimal morphism is ramified if and only if it is subintegral. Let $\{R_1,\ldots, R_n\}$ be finitely many infra-integral extensions of a ring $R$. It is easy to show that $R\to \prod_{i=1}^n R_i$ is infra-integral. But this result is no longer valid for subintegrality.

A ring extension $R\subseteq S$ is called {\it t-closed} if $b\in S,\ r\in R,\ b^2-rb,b^3-rb^2\in R \Rightarrow b\in R$ \cite{Pic 2}. Now, $R\subseteq S$ is called {\it seminormal} if $b\in S,\ b^2,b^3\in R \Rightarrow b\in R$ \cite{S}. If $R\subset S$ is seminormal, $(R:S)$ is a radical ideal of $S$. The $t$-{\it closure}  $\substack{t\\ S}R$ (resp$.$  {\it seminormalization}  $\substack{+\\ S}R$)     of $R$ in $S$ is the smallest $B \in [R,S]$  such that $B\subseteq S$ is t-closed (resp$.$  seminormal). Moreover,  $\substack{t\\ S}R$ (resp$.$  $\substack{+\\ S}R$) is the greatest $B\in[R,S]$ such that $R\subseteq B$ is infra-integral (resp$.$ subintegral). The chain $R\subseteq \substack{+\\ S}R\subseteq \substack{t\\ S}R \subseteq S$ is called the {\it canonical decomposition} of $R\subseteq S$.

T-closure and seminormalization  both commute with localization at arbitrary multiplicatively closed subsets (\cite[Proposition 3.6]{Pic 1}, \cite[Proposition 2.9]{S}).

According to J. A. Huckaba and I. J. Papick \cite{HP}, an extension $R\subseteq S$ is termed a $\Delta_0$-{\it extension} provided each $R$-submodule of $S$ containing $R$ is an element of $[R,S]$,  a {\it quadratic extension} if $R+Rt\in [R,S]$ for each $t\in S$, and  a $\Delta$-{\it extension} if $R_1+R_2\in[R, S]$ for $R_1,R_2\in[R,S]$. By \cite[Proposition 5]{HP}, an extension is  $\Delta_0$ if and only if this extension is quadratic and  $\Delta$. We recall here for later use an unpublished result of the Gilbert's dissertation.

\begin{proposition}\label{1.4} \cite[Proposition 4.12]{Gil} Let $R \subseteq S$ be a ring extension with conductor $I$ and such that $S = R + Rt$ for some $t\in S$. Then the $R$-modules $R/I$ and $S/R$ are isomorphic. Moreover, each of the $R$-modules between $R$ and $S$ is a ring (and so there is a bijection from $[R,S]$  to the set of ideals of $R/I$). 
\end{proposition}

We end this introduction with a new result that introduces and gives the flavor of the next section. 

\begin{proposition}\label{1.5} Let $R$ be a commutative ring and $n\geq 2$ a positive integer. Then $R\subseteq R^n$ has FCP if and only if $R$ is an Artinian ring. 
 \end{proposition}

\begin{proof} Assume that $R\subseteq R^n$ has FCP and that there is an infinite chain $\{I_j\}_{j\in J}$ of ideals of $R$. For each $j\in J$, set $S_j:=R+(0\times I_j)$. Then, $\{S_j\}_{j\in J}$ is an infinite chain of $R$-subalgebras of $R^n$, which is absurd. Hence, any chain of ideals of $R$ is finite and $R$ is  Artinian.

Conversely,   $R\subseteq R^n$  has a zero conductor and $R^n$ is  f.g$.$ over $R$. Thus $R\subseteq R^n$ has FCP in view of \cite[Theorem 4.2]{DPP2}, if $R$ is Artinian. 
\end{proof}

The following results will be useful.

\begin{proposition}\label{1.6} Let $(R,M)$ be a local Artinian ring such that $R/M$ is infinite and $R\subseteq S$ a ring extension with conductor $C:= (R:S)$.
\begin{enumerate}
\item If $R\subset S$ has  FIP  and  is subintegral, then $[R,S]$ is linearly ordered.

\item If $R\subseteq S$ is finite, seminormal and infra-integral, then  $R\subseteq S$ has FIP.

\item If $R\subset S$ is  finite and  infra-integral, then $R\subset S$ has FIP if and only if $R\subseteq \substack{+\\ S}R$ has FIP. 
\end{enumerate} 
\end{proposition}

\begin{proof} (1)  There is no harm to assume that   $C=0$ because the map $[R,S] \to [R/C,S/C]$ defined by $T \mapsto T/C$ is  bijective.  If $R$ is not a field, then  the proof of \cite[Proposition 5.15]{DPP2} shows that  $[R,S]$ is linearly ordered. 

Now, assume  that $R$ is a field, so that $0=(R:S)$ and $R$ is infinite. Since $R\subset S$ is an FIP subintegral extension, $S$ is Artinian local and not a field with $\{N\}:=\mathrm{Max}(S)$, because $R\cong S/N$ by subintegrality shows that $N\neq 0$. From \cite[Theorem 3.8]{ADM}, we get that $S=R [\alpha]$, for some $\alpha\in S$ such that $\alpha^3=0$. In view of the proof of \cite[Lemma 3.6(b)]{ADM}, $[R,S]$ is linearly ordered.

(2) We can assume that $R\neq S$ and $C= 0$ by considering $R/C \to S/C$ and using \cite[Proposition 3.7(c)]{DPP2}. By \cite[Proposition 5.16]{DPP2}, we get that  $R\subset S$ has FIP.

(3)  Assume that $R\subset S$ is  finite  and infra-integral  and set $T:=\substack{+\\ S}R$. Then, $T$ is  local Artinian with maximal ideal  $N$ and $T/N\cong R/M$ is infinite. Moreover, $T\subseteq S$ is  finite, seminormal, infra-integral  and  has FIP by (2).

If $R\subset S$ has FIP, then $R\subseteq T$ has FIP. Conversely, assume that  $R\subseteq T$ has FIP. In view of \cite[Theorem 5.8]{DPP2},  $R\subset S$ has FIP.
\end{proof}

We will use the following result. If $R_1,\ldots,R_n$ are finitely many rings,  the ring $R_1\times \cdots \times R_n$ localized at the prime ideal $P_1\times R_2\times\cdots \times R_n$ is isomorphic to $(R_1)_{P_1}$ for $P_1 \in \mathrm{Spec}(R_1)$. This rule works for any prime ideal of the product. 

\section{FCP or FIP extensions for products of rings}

We extract from the more precise result \cite[Proposition 4.15]{DPP3} the following statement, about the canonical diagonal extension $K\subseteq K^n$, for a field $K$  and a positive integer $n> 1$.  Recall that the $n$th {\it Bell number} $B_n$ is the number  of partitions of $\{1,\ldots,n\}$ \cite[p. 214]{An}. Actually, the finiteness of $|[K,K^n]|$ comes from \cite[Proposition 3, p. 29]{Bki A}

\begin{proposition}\label{2.1} Let $K$ be a field and $n$ a positive integer, $n>1$.  Then
$|[K,K^n]|=B_n$, where $B_n$ is the $n$th  Bell number and $K\subseteq K^n$ is a seminormal and infra-integral FIP extension.
 
 \end{proposition}
 
We now intend to extend the above result to diagonal ring extensions $\delta_n:R\hookrightarrow R^n$, for arbitrary rings $R$. We need information about some closures and give necessary conditions for the FCP or FIP properties hold. If $R\subseteq R_i,\ i=1,\ldots,n,\ n\geq 2$ are finitely many ring extensions and $\delta:R\hookrightarrow\prod_{i=1}^nR_i$ is the canonical diagonal extension, it can be factored $R \hookrightarrow R^n\hookrightarrow\prod_{i=1}^nR_i$. We can also consider that $R\hookrightarrow R^2 $ is a subextension by considering the product $R\times R\to R_1\times\prod_{i=2}^nR_i$ of the extensions $R\hookrightarrow R_1$ and $R\hookrightarrow\prod_{i=2}^nR_i$. Of course, this embedding of $R^2$ is not unique. A more complete study appears in Section 4 (see Proposition 4.6).

\begin{proposition}\label{2.2} Let $R\subseteq R_i,\ i=1,\ldots,n,\ n\geq 2$ be finitely many ring extensions, $\mathcal{R}: = \prod_{i=1}^nR_i$  and $R \subseteq \prod_{i=1}^nR_i= \mathcal{R}$ the canonical diagonal extension.  Then:

\begin{enumerate}
\item $\mathrm{Supp}(\mathcal{R}/R)=\mathrm{Spec}(R)$.

\item Assume that $R\subseteq \mathcal{R}$ has FCP (resp. FIP). Then, $R$ is an Artinian ring and each extension $R \subseteq R_i$ has FCP (resp. FIP).

\item Assume that $R\subseteq \mathcal{R}$ has FIP. Then, $R$ has finitely many ideals.
\end{enumerate}
\end{proposition}

\begin{proof} 
 We have $R^2 \subseteq \prod_{i=1}^n R_i$ and $R^n \subseteq \prod_{i=1}^n R_i$.
 
(1) Let $P\in\mathrm{Spec}(R)$. Then, $R_P\neq 0$ implies $(1,0)\not\in R_P$ and $P\in\mathrm {Supp}(R^2/R)\subseteq\mathrm{Supp}(\mathcal{R}/R)$, which gives (1). Indeed, $(R^2/R)_P\cong (R_P)^2/R_P$. 

(2) Assume that $R\subseteq \mathcal{R}$ has FCP, so that $R\subseteq R^n$ has FCP. Then, $R$ is an Artinian ring in view of Proposition~\ref{1.5}. Statements about FCP or FIP are clear.

(3) Assume that $R\subseteq \mathcal{R}$ has FIP, so that $R\subseteq R^2$ has FIP. Let $I,J$ be two distinct ideals of $R$. Then, $R+(0\times I)$ and $R+(0\times J)$ are two distinct $R$-subalgebras of $R^2$. Since $R\subseteq R^2$ has FIP, it follows that $R$ has finitely many ideals.
\end{proof} 

Rings which have finitely many ideals are characterized by   D. D. Anderson and S. Chun \cite{AC}, a result that   will be often used.

\begin{proposition}\label{2.3} \cite[Corollary 2.4]{AC} A commutative ring $R$ has only finitely many ideals if and only if $R$ is a finite direct product of finite local rings, SPIRs, and fields, that are the local rings of $R$.
 \end{proposition}
 
 From now on, a ring $R$ with finitely many ideals is termed an FMIR and a $\Sigma$FMIR if at least a local ring of $R$ is an infinite SPIR. We also call $\Sigma$PIR an infinite SPIR. For an arbitrary ring $R$, we denote by $\Sigma\mathrm{Max}(R)$ the set of all $M\in\mathrm{Max}(R)$ such that $R_M$ is an infinite FMIR.
 
\begin{proposition}\label{2.4} Let $R\subseteq R_i,\ i=1,\ldots,n$ be finitely many ring extensions and $\mathcal{R}:=\prod_{i=1}^nR_i$. Let $\overline R_i$ (resp$.$ $\overline{\mathcal{R}}$) be the integral closure of $R$ in $R_i$ (resp$.$ $ \mathcal{R}$). Then:

\begin{enumerate}
\item $\overline{\mathcal{R}}= \prod_{i=1}^n\overline R_i$.

\item Assume that $R\subseteq R_i$ has FCP for each $i$. Then, $\overline{\mathcal{R}}\subseteq   \mathcal{R}$ has FCP (and FIP). 
\end{enumerate}
 \end{proposition}

\begin{proof} (1) is \cite[Proposition 9, ch. V, p. 16]{Entiers}. 

(2) Assume that $R\subseteq R_i$ has FCP for each $i$. In view of \cite[Theorem 3.13]{DPP2}, we get that $\overline R_i\subseteq R_i$  has FCP for each $i$. This extension has also FIP since FCP and FIP are equivalent for an integrally closed extension \cite[Theorem 6.3]{DPP2}. Now, use \cite[Proposition III.4]{DMPP}, to get that $\prod_{i=1}^n\overline R_ i \subseteq \prod_{i=1}^n R_i$ has FCP (and then FIP because integrally closed).
  \end{proof} 

\begin{corollary}\label{2.5} Let $R\subseteq R_1$ and $R\subseteq R_2$ be two integrally closed extensions. Then, $R\subseteq R_1\times R_2$ has FCP (resp$.$ FIP) if and only if each $R\subseteq R_i$  has  FCP  and $R$ is  Artinian (resp$.$ an FMIR). 

In particular,   $R\subseteq R^2$ has FIP if and only if $R$  is an FMIR. 
 \end{corollary}

\begin{proof} One implication is obvious, since any $R$-subalgebra $S_1$ of $R_1$ yields an $R$-subalgebra $S_1\times R_2$ of $R_1\times R_2$. Then, use Proposition~\ref{2.2}.

Conversely, assume that $R\subseteq R_1$ and $R\subseteq R_2$ have both FCP (and then FIP) and that $R$ is  Artinian. Then, $R^2\subseteq R_1\times R_2$ has FCP (resp$.$ FIP) by Proposition~\ref{2.4}. Moreover, $R^2\subseteq R_1\times R_2$ is integrally closed and $R\subseteq R^2$ is an integral extension. In view of Proposition~\ref{1.5}, it follows that $R\subseteq R^2$ and so $R\subseteq  R_1\times R_2$ have FCP by \cite[Theorem 3.13]{DPP2}.

Now, assume that $R\subseteq R_1$ and $R\subseteq R_2$ have both FIP and that $R$ is an FMIR.  By Proposition~\ref{1.4}, $R\subseteq R^2$ as well as $R\subseteq R_1\times R_2$ have FIP by \cite[Theorem 3.13]{DPP2}.
\end{proof} 

\begin{proposition}\label{2.6} Let $R\subseteq R_i,\ i=1,\ldots,n$, be finitely many integral extensions,  $S_i:= \substack{+\\ R_i}R,\ T_i:= \substack{t\\ R_i}R$ for each $i$,  $\mathcal{R}:= \prod_{i=1}^n R_i,\ \mathcal{S}:= \prod_{i=1}^n S_i$ and $\mathcal T : = \prod_{i=1}^n T_i$. Then:

\begin{enumerate}
\item  $\substack{+\\ \mathcal{R}}R=\substack{+\\ \mathcal{S}}R$ and $\substack{t\\ \mathcal{R}}R=\mathcal T$. 

\item If each $T_i\subseteq R_i$ has FCP (resp$.$ FIP), then $\substack{t\\ \mathcal{R}}R\subseteq  \mathcal{R}$ has FCP (resp$.$ FIP).  This holds if each $R\subseteq R_i$ has FCP (resp$.$ FIP). 
\end{enumerate}
 \end{proposition}

\begin{proof} (1)  Obviously, $\substack{+\\ \mathcal{S}}R\subseteq\substack{+\\ \mathcal{R}}R$ and is subintegral. Moreover, $\mathcal{S}\subseteq\mathcal{R}$ is seminormal, since so are each $S_i\subseteq R_i$. Then, $\mathcal{S}\in[\substack{+\\ \mathcal{R}}R,\mathcal{R}]$, with $\substack{+\\ \mathcal{R}}R\subseteq\mathcal{S}$ seminormal, so that $\substack{+\\ \mathcal{S}}R\subseteq\substack{+\\ \mathcal{R}}R$ is also seminormal, then an equality.

We know that $\prod_{i=1}^n T_i\subseteq\prod_{i=1}^n R_i$ is $t$-closed \cite[Lemma 5.6]{GP1}. To conclude, it is enough to show that $R \subseteq \prod_{i=1}^n T_i$ is infra-integral.

The prime ideals of $\prod_{i=1}^nT_i$ are the $P_i\times\prod_{j=1,j\neq i}^nT_j$, where $P_i$ is a prime ideal of $T_i$. For $P_i\in\mathrm{Spec}(T_i)$, set $Q _i:=P_i\cap R$. Then, $(\prod_{i=1}^nT_i)/(P_i\times\prod_{j=1,j\neq i}^nT_j)\cong T_i/P_i\cong R/Q_i$, since $R\subseteq T_i$ is infra-integral. It follows that $R\subseteq \prod_{i=1}^nT_i$ is infra-integral

(2)  In view of \cite[Proposition 3.7(d)]{DPP2}, we get that $\prod_{i=1}^nT_ i=\substack{t\\ \mathcal{R}}R \subseteq\mathcal{R}$ has FCP (resp$.$ FIP). There was a misprint in the statement of \cite[Proposition 3.7(d)]{DPP2}, where we should read: If $R=R_1\times\cdots\times R_n$ is a finite product of rings and $R\subseteq S$ satisfies FCP, then $S$ can be identified with a product of rings $S_1\times\cdots\times S_n$ where $R_i \subseteq S_i$ for each $i$. Then $\ell[R,S]=\sum_{i=1}^n \ell[R_i,S_i]$.
\end{proof} 

The next proposition  and Proposition~\ref{2.2} enables us to reduce our study to quasi-local rings.

\begin{proposition}\label{2.7} \cite[Proposition 3.7 and Corollary 3.2]{DPP2} Let $R\subseteq S$ be a ring extension.

\begin{enumerate}

\item If $R\subseteq S$ has FCP (FIP), then  $|\mathrm {Supp}(S/R)|<\infty$.

\item If  $|\mathrm {MSupp}(S/R)|<\infty$, then $R\subseteq S$ has FCP (FIP) if and only if $R_M\subseteq S_M$ has FCP (FIP) for each $M\in\mathrm {MSupp}(S/R)$.

\end{enumerate}
\end{proposition}

\begin{proposition}\label{2.8} Let $R\subseteq R_i,\ i=1,\ldots,n$, be finitely many subintegral  extensions and $\mathcal R: = \prod_{i=1}^n R_i$, where $(R,M)$ is a quasi-local ring. Then:

\begin{enumerate}
\item Each $R_i$ is a quasi-local ring with $\{N_i\}: = \mathrm{Max}(R_i)$ and $R\subseteq \mathcal R$ is infra-integral. 

\item  Set $ N:=\prod_{i=1}^nN_i$ and $S:=R+N$. Then $(S, N)$ is a quasi-local ring and $\mathrm{Spec}(S)=\{P'_i\times\prod_{j=1,j\neq i}^nN_j\mid P'_i\in\mathrm{Spec}(R_i),i=1,\ldots,n\}$. In particular, $R\subseteq S$ is infra-integral and $ \substack{+\\\mathcal R}R \subseteq S$.                  

\item Assume $\dim(R)=0$. Then, $\substack{+\\\mathcal R}R =S$. 

\item If each $R_i$ is a Noetherian ring and a f.g$.$  $R$-module, then $S$ is a f.g$.$ $R$-module. 
\end{enumerate}
 \end{proposition}

\begin{proof} (1) $R_i$ is quasi-local since $R\subseteq R_i$ is subintegral (Definition~\ref{1.3}). Now, an  arbitrary prime ideal of $\mathcal R$ is of the form $P':=P'_i\times\prod_{j=1,j\neq i}^nR_j$, for some $i$ and $P'_i\in\mathrm{Spec}(R_i)$. Setting $P:=P'\cap R$, we see that $P=P'_i\cap R$. From $\mathcal R/P'\cong R_i/P'_i\cong R/P$, since $R\subseteq R_i$ is subintegral, we deduce that $R\subseteq\mathcal R$ is infra-integral. 

(2) The ideals $N'_i:=N_i\times\prod_{j=1,j\neq i}^nR_j$ are the maximal ideals of $\mathcal R$, for $i\in\{1,\ldots,n\}$, and they all lie over $M$.  Observe that $S$ is an $R$-subalgebra of $\mathcal R$. From $N\cap R=M$, we infer that $S/N\cong R/M$ and that $N \in \mathrm{Max}(S)$. Since $R\subseteq\mathcal R$ is an integral extension, so is $S\subseteq\mathcal R$. Moreover, each $N'_i$ lies over $N$. Hence $(S,N)$ is a quasi-local ring. 

Let $Q\in\mathrm{Spec}(S)$, there is some $P\in\mathrm{Spec}(\mathcal R)$ lying over $Q$,  of the form  $P:=P'_i\times\prod_{j= 1,j\neq i}^nR_j$, for some $P'_i\in\mathrm{Spec}(R_i)$. Since $Q\subseteq N$, we get $Q\subseteq(P'_i\times\prod_{j=1,j\neq i}^nR_j)\cap (\prod_{k=1}^nN_k)=P'_i\times\prod_{j=1,j\neq i}^nN_j\subseteq S\cap P=Q$, so that $Q=P'_i\times \prod_{j=1,j\neq i}^nN_j$. Conversely, any ideal of the form $P'_i\times\prod_{j=1,j\neq i}^nN_j$, for some $i$ and  $P'_i\in\mathrm{Spec}(R_i)$ is in $\mathrm{Spec}(S)$, since $P'_i\times \prod_{j=1,j\neq i}^nR_j$ lies over it. 

Since $R\subseteq S$ is a subextension of $R\subseteq\mathcal R$, (1) entails that $R\subseteq S$ is infra-integral. But $\prod_{i=1}^nN_i$ is also an ideal of $\mathcal R$, so that $N=(S:\mathcal R)$. To end, $\mathcal R/N\cong(R/M)^n$ and $S/N\cong R/M$ give that $S/N\subseteq\mathcal R/N$ is seminormal by Proposition~\ref{2.1}, and so is $S\subseteq\mathcal R$. Then, $\substack{+\\\mathcal R}R \subseteq S$.

(3) Assume $\dim(R)=0$. Then, $\mathrm{Spec}(S)=\{\prod_{i=1}^nN_i\}=\{N\}$. Then $S/N\cong R/M$ shows that $R\subseteq S$ is a subintegral extension and $S= \substack{+\\\mathcal R}R $.

(4) If each $R_i$ is  Noetherian and f.g$.$ over $R$, then, each $N_i$ is a f.g$.$ $R_i$-module, and also a f.g$.$ $R$-module. Hence, $R+N$ is a f.g$.$ $R$-module.
\end{proof} 

\begin{remark}\label{2.9} Contrary to the t-closure, the seminormalization of a diagonal morphism is not the product of the seminormalizations. We can compare these results with \cite[Lemma 5.6]{GP1}, which says that seminormalization and t-closure commute with finite products of morphisms. 
\end{remark}

\begin{proposition}\label{2.10} Let $R\subseteq R_i,\ i=1,\ldots,n$ be finitely many  integral extensions and $\mathcal R := \prod_{i=1}^nR_i$, where $(R,M)$ is a quasi-local ring. Then:

\begin{enumerate}
\item  $\substack{t\\\mathcal R}R \subseteq \mathcal R$ has FCP (resp$.$FIP) if each $R\subseteq R_i$ has.

\item If $\dim(R)=0$ and each $R\subseteq R_i$ has FCP, then, $\substack{+\\\mathcal R}R \subseteq \substack{t\\\mathcal R}R $  has FIP.

\item If $\dim(R)=0$ and  each  $R\subseteq R_i$ has FCP (resp$.$ FIP), then $R\subseteq \mathcal R$ has FCP (resp$.$ FIP) if and only if $R\subseteq{}_{\mathcal R}^+R$ has FCP (resp$.$ FIP). 
\end{enumerate}
 \end{proposition}

\begin{proof} (1) Proposition~\ref{2.6} gives that $\substack{t\\\mathcal R}R \subseteq \mathcal R$ has FCP (resp$.$ FIP). 

(2) Set $T_i:=\substack{t\\ R_i}R,\ S_i:=\substack{+\\ R_i}R,\ T:=\prod_{i=1}^nT_i=\substack{t\\\mathcal R}R$. Now, each $R\subseteq S_i$ is subintegral. It follows from Proposition~\ref{2.8} and \cite[Lemma 5.6]{GP1} that $S:=R+\prod_{i=1}^nN_i=\substack{+\\\mathcal R}R$, where $N_i$ is the maximal ideal of $S _i$ for each $i$. Moreover, $N_i\subseteq(S_i:T_i)$ holds for each $i$ by \cite[Proposition 4.9]{DPP2} and $S_i$ and $T_i$ share the ideal $N_i$, since $S_i\subseteq T_i$ is seminormal and infra-integral. Actually, $N_i=(S_i:T_i)$ when $S_i\neq T_i$ and $(S_i:T_i)=S_i$ when $S_i=T_i$. Therefore we get $ N:=\prod_{i=1}^nN_i\subseteq(S:T)$ and $N$ is a common ideal of $S$ and $T$, maximal in $S$ by Proposition~\ref{2.8}. Set $k:=R/M\cong S/N\cong S_i/N_i\cong T_i/N_{i,j}$, for each maximal ideal $N_ {i,j}$ of $T_i $. For each $i$, we have $N_i=\cap_{j=1}^{n_i}N_{i,j}$, for some $n_i$, \cite[Proposition 4.9]{DPP2}, so that $T_i/N_i\cong\prod_{j=1}^{n_i}T_i/N_{i,j}$. Then the extension $S/N\subseteq(\prod_{i=1}^nT_i)/N\cong\prod_{i=1}^n(T_i/N_i)$ can be identified to $k\subseteq k^{\sum n_i}$, which has FIP (and then FCP) by Proposition~\ref{2.1}. It follows that ${}_{\mathcal R}^+R\subseteq{}_{\mathcal R}^tR$ has FIP (and then FCP) by \cite[Proposition 3.7]{DPP2}. 

(3) By \cite[Theorem 4.6 and Theorem 5.8]{DPP2}, $R\subseteq \mathcal R$ has FCP (resp$.$ FIP) if and only if $R\subseteq \substack{+\\\mathcal R}R ,\ \substack{+\\\mathcal R}R \subseteq \substack{t\\\mathcal R}R $ and $\substack{t\\\mathcal R}R \subseteq \mathcal R$ have FCP (resp$.$ FIP) if and only if $R\subseteq \substack{+\\\mathcal R}R $ has FCP (resp$.$ FIP) by (1) and (2). 
\end{proof} 

The FCP case is now completely solved with the following theorem.

\begin{theorem}\label{2.11} Let $R\subseteq R_i,\ i=1,\ldots,n,\ n\geq 2$ be finitely many  extensions and $\mathcal R:=\prod_{i=1}^nR_i$. Then $R\subseteq\mathcal R$ has FCP if and only if $R$ is an Artinian ring and each extension $R \subseteq R_i$ has FCP. 
 \end{theorem}

\begin{proof} The``only if"  implication is  Proposition~\ref{2.2}(2).

Conversely, assume that $R$ is an Artinian ring and each $R \subseteq R_i$ has FCP. From Proposition~\ref{2.4}, we infer that $\overline {\mathcal R}\subseteq\mathcal R$ has FCP. Moreover $R^n\subseteq\overline{\mathcal R}=\prod_{i= 1}^n\overline{R}_i$ has FCP by \cite[Proposition 3.7]{DPP2} and $R\subseteq R^n$ has FCP by Proposition~\ref{1.5}, giving that $R\subseteq\overline{\mathcal R}$ has FCP by \cite[Corollary 4.3]{DPP2}. To end, use \cite[Theorem 3.13]{DPP2} to get that $R\subseteq \prod_{i=1}^nR_i$ has FCP.
\end{proof} 

We  now consider the FIP property for the product    of two FIP extensions. The case of $n>2$ FIP extensions is studied in Section 4.

\begin{proposition}\label{2.12} Let $R\subset R_1,R_2$ be two subintegral FIP extensions and set $ \mathcal R:=R_1\times R_2$. Assume that $(R,M)$ is quasi-local such that $|R/M|=\infty$. Then $R\subseteq \mathcal R$ has not FIP. 
 \end{proposition}

\begin{proof} Let $N_i$ be the maximal ideal of $R_i$. The infra-integrality of $R\subset R_i$ implies that $M\neq N _i$. It follows that $S_1:=R+(N_1\times M)$ and $S_2:=R+(M\times N_2)$ are incomparable $R$-subalgebras of $S:=R+(N_1\times N_2)$, because $(x, 0)\in S_1\setminus S_2$ for $x \in N_1\setminus M$ and $(0,y) \in S_2\setminus S_1$ for $y\in N_2\setminus M$.

Assume now that $R\subset\mathcal R$ has FIP. In this case, $R\subset S$ has FIP and $R$ is Artinian by Proposition~\ref{2.2}. It follows that $S={}_{\mathcal R}^+R$ by Proposition~\ref{2.8}, so that $R\subset S$ is a subintegral extension. From Proposition~\ref{1.6}, we deduce that $S_1$ and $S_2$ are comparable, a contradiction and $R\subset \mathcal R$ has not FIP. 
\end{proof} 

In order to settle the main Theorem~\ref{2.17} of the section, we begin to clear the way by studying when $R\subseteq \mathcal R$ has not FIP. We can suppose that $R_1=R$, because $R\times R_2 \subseteq R_1\times R_2$. By Proposition~\ref{2.2} and Proposition~\ref{2.3}, we need only to consider a $\Sigma$PIR $(R,M)$ in view of \cite[Proposition 3.7]{DPP2}. Indeed, the case of a field $R$ has already been studied in \cite{ADM}. Note that if $(R,M)$ is a local Artinian ring, then $R$ is finite if and only if $R/M$ is finite, since $M^n=0$ for some integer $n$. In such a case, any finite extension of $R$ has FIP. We first look at  minimal ramified extensions.  Before, we give a useful lemma.     

\begin{lemma}\label{2.13} Let $R\subset S$ be a ring extension, where $(R,M)$ is a  quasi-local ring with $|R/M|=\infty$. Let $\mathcal{F}$ be a set of representative elements of $R/M$. If there exists a family $\{R_{\alpha}\}$ of elements of $[R,S]$ such that $R_{\alpha}\neq R_{\beta}$ for each $\alpha\neq \beta\in\mathcal{F}$, then $R\subset S$ has not FIP.
 \end{lemma}

\begin{proof} Obvious.
\end{proof} 

\begin{lemma}\label{2.14} Let $R\subset S$ be a minimal ramified extension, where $(R,M)$ is a SPIR. 

\begin{enumerate}
\item There exists $t\in M$ such that $M=Rt$ and $t^p=0$, with $t^{p-1}\neq 0$, for some integer $p> 1$. 
\item Let $N$ be the maximal ideal of $S$. There exists $x\in S\setminus R$ such that $S=R+Rx,\ N= Rt+Rx$. Moreover, there are  some unique positive integers $p \geq k,q\geq 1$ and some $a,b\in R\setminus M$ such that  $x^2=at^k,\ tx=bt^q$. Then, $(R:_Rx)=Rt= M = (R:S)$.

\item  $q\geq 2$ holds. 
\end{enumerate}
\end{lemma}

\begin{proof} (1) is the definition of a SPIR (see Section 1). Each element of $R$ is of the form $ut^h$ for some unique integer $h \leq p$ and some unit $u$.

(2) The integers $k$ and $q$ exist by Theorem~\ref{1.2} (c) or \cite[Theorem 2.3 (c)]{DPP2} because $x^2, tx \in M$ and are  unique by (1) since the ideals of $R$ are linearly ordered.

(3) Assume $q=1$. Then, $tx=bt$ implies $t(x-b)=0$. But $x-b\not\in N$ since $b\in R\setminus M$, so that $x-b$ is a unit in $S$, and then $t=0$, a contradiction, which yields $q\geq 2$. In particular, $tx\in Rt^2$. 
\end{proof} 

\begin{proposition}\label{2.15} Let $R\subset S$ be a minimal ramified extension, where $(R,M)$ is a $\Sigma$PIR. We set $\mathcal R := R \times S$ and  $\{N\}:= \mathrm{Max}(S)$.

\begin{enumerate}
\item $T:=\substack{+\\\mathcal R}R =R+(M\times N)$.

\item $R\subset \mathcal R$ has FIP if and only if $N^2=M$ and $MN=M^2=0$.
\end{enumerate}
\end{proposition}

\begin{proof} (1) The value of $T$ is given in Proposition~\ref{2.8}.
 
(2) We keep the notation of Lemma~\ref{2.14}. There exists $t\in M$ such that $M=Rt$ and $t^p=0$, with $t^{p-1}\neq 0$, for some integer $p> 1$. There exists $x\in S\setminus R$ such that $S=R+Rx,\ N=R t+Rx$. Moreover, there are some positive integers $p\geq k,q\geq 1$ and some $a,b\in R\setminus M$ such that $x^2=at^k,\ tx=bt^q$, with $q\geq 2$. Then, $M^2=Rt^2,\ MN=Rt^2+Rtx=Rt^2$ since $tx\in Rt^2$, so that $M^2=MN$, and $N^2=Rt^2+Rtx+Rx^2=Rt^2+Rt^k$. 

Let $\mathcal{F}$ be a set of representative elements of $R/M$. Then $\mathcal{F}$ is infinite.

Assume first that $k>1$, so that $x^2\in Rt^2$. For $\alpha\in\mathcal{F}$, set $R_{\alpha}:=R+R(0,t+\alpha x)+R(0,t^2)$. Then, $R_{\alpha}\in[R,T]$. Let $\beta\in\mathcal{F}$ be such that $\alpha\neq\beta$, so that $\alpha-\beta\not\in M$. Assume that $R_{\alpha}=R_{\beta}$. We get that $(0,t+\alpha x)=(c,c)+(0,dt+d\beta x)+(0,et^2)$, for some $c,d,e\in R$, giving $0=c$ and $t+\alpha x=c+dt+d\beta x+et^2=dt+d\beta x+et^2$. Since $(\alpha-d\beta)x=(d-1)t+et^2\in M$, we get $\alpha-d\beta\in M\ (*)$ in view of Lemma~\ref{2.14}(2). It follows that there exists $d'\in R$ such that $\alpha-d\beta=d't$, yielding $d'tx=d'b t^q=(d-1)t+et^2$, so that $(d-1)t=d'bt^q-et^2\in Rt^2$, leading to $d-1\in M\ (**)$. But $(*)$ and $(**)$ give $\alpha-\beta\in M$, a contradiction. Then, $R_{\alpha}\neq R_{\beta}$, and $R\subset\mathcal R$ has not FIP in view of Lemma~\ref{2.13}.

It follows that when $R\subset \mathcal R$ has FIP, we must have $k=1$. 

Now, assume that $k=1$. Then, $x^2t=at^2=(tx)x=xbt^q=(xt)bt^{q-1}=b^2t^{2q-1}$, so that $at^2-b^2t ^{2q-1}=t^2(a-b^2t^{2q-3})=0$. But $q\geq 2$ implies $2q-3\geq 1$, giving $a-b^2t^{2q-3}$ is a unit in $ R$. Then, $t^2=0$ and $p=q=2$, with $tx=0$.  

So, when $R\subset\mathcal R$ has FIP, then $k=1$ and $p=q=2$, which give $M^2=MN=0$ and $N^2=Rt=M$. 

Assume now that $N^2=M$ and $MN=M^2=0$. Then, $Rt=Rt^2+Rt^k$, giving $k=1$, and $Rt^2=0$, giving $p=q=2$. Observe that $R\subset\mathcal R$ is an integral FCP extension by Proposition~\ref{2.11}. Using notation and statement of \cite[Theorem 5.18]{DPP2}, set $R_1:=R+TM=R$. Then, $ T=R[(0,x)],\ (0,x)^3=0\in M$, and, with $T':=R[(0,x)^2]=R[(0,t)]$ and $T'':=R+T'M=R$, we have $T'= T''[(0,t)]$, with $(0,t)\in T$, and $(0,t)^3=0\in T'M$. We can conclude that $R\subset\mathcal R$ has FIP.  
\end{proof} 

\begin{corollary}\label{2.16} Let $R\subset S$ be a non minimal subintegral FIP extension, where $(R,M)$ is a $\Sigma$PIR. Then, $R\subset R\times S$ has not FIP.
 \end{corollary}

\begin{proof} Since $R\subset S$ has FIP, there is $S_1\in[R,S]$, such that $R\subset S_1$ is a minimal extension, necessarily ramified. Assume that $R\subseteq R\times S$ has FIP, then so has $R\subset R\times S_1$. Using the notation of Lemma~\ref{2.14} and Proposition~\ref{2.15} for $R\subseteq S_1$, we have $M=Rx^2$, $S_1=R+Rx$, $N=Rx^2+Rx$, where $N$ is the maximal ideal of $S_1$ and $x^3=0 $, $x^2\neq 0$. There exists $S_2\in[S_1,S]$ such that $S_1\subset S_2$ is a minimal extension, necessarily ramified. Let $P$ be the maximal ideal of $S_2$. In view of \cite[Theorem 2.3(c)]{DPP2}, there is $y\in S_2$ such that $S_2=S_1+S_1y=R+Rx+Ry+Rxy$ and $P=N+S_1y=Rx^2+Rx+Ry+Rxy$. Moreover, $(S_1:y)=N$. But, $NP\subseteq N$ gives $xy\in N$ and $P^2\subseteq N$ gives $y^2\in N$, so that $P=Rx^2+Rx+Ry$ and there exist $b,c,d,e\in R$ such that $y^2=bx^2+cx\ (*)$ and $yx=dx^2+ex\ (**)$. It follows that $yx^2=x(dx^2+ex)=ex^2$, so that $(y-e)x^2=0$. If $e\not\in M$, then $e\not\in P$ and $e-y$ is a unit in $S_2$, giving $x^2=0$, a contradiction. But $e\in M$ implies that $e x^2\in Rx^4=0$, so that $yx^2=0$. Now, $(*)$ gives $xy^2=bx^2x+cx^2= dx^2y+exy=cx^2$. But $e\in M=Rx^2$ entails $ex\in Rx^3=0$, so that $xy^2=dx^2y=0$, whence $cx^2=0$, from which we infer that $c\in M=Rx^2$. Therefore, we get $y^2=bx^2$ since $x^3=0$. Let $\mathcal{F}$ be a set of representative elements of $ R/M$. For $\alpha\in\mathcal{F}$, set $R_{\alpha}:=R+R(0,x+\alpha y)+R(0,x^2)$. Then, $R_{\alpha}\in [R,R+(R\times S_2)]$ since $(x+\alpha y)^2=(1+2\alpha d+\alpha^2b)x^2$. Let $\beta\in\mathcal{F}$ be such that $\alpha\neq\beta$, so that $\alpha-\beta\not\in M$. Assume that $R_{\alpha}=R_{\beta}$. We get that $(0,x+\alpha y)=(c,c)+(0,dx+d\beta y)+(0,ex^2)$, for some $c,d,e\in R$, giving $0=c$ and $x+\alpha y=c+dx+d\beta y+ex^2=dx+d\beta y+ex^2$. Since $(\alpha-d\beta)y=(d-1)x+ex^2\in N$, we get $\alpha-d\beta\in N\cap R=M\ (\dag)$. It follows that there exists $d'\in R$ such that $\alpha-d\beta=d'x^2$, yielding $0=d'x^2y=(d-1)x+ex^2$, so that $(d-1)x\in M$, leading to $d-1\in M\ (\dag\dag)$. But $(\dag)$ and $(\dag\dag)$ give $\alpha-\beta\in M$, a contradiction. Then, $R_{\alpha}\neq R_{\beta}$, and $R\subset R\times S$ has not FIP in view of Lemma~\ref{2.13}.
\end{proof} 

To shorten, a minimal ramified (subintegral) extension $(R,M)\hookrightarrow(S,N)$ between quasi-local rings is called {\it special} if $M^2=MN=0$ and $N^2=M$, as in Proposition~\ref{2.15}. Such extensions exist. Any minimal ramified extension $R\subset S$ such that $R$ is a field is special. Here is another example. Let $K$ be a field and $R:=K[T]/(T^2)$. If $t$ is the class of $T$ in $R$, let $S:=R[X]/(X^2-t,Xt)$. The natural map $R\to S$ is injective. This follows from the fact that $R[X]$ is a free $K[X]$-module with basis $\{1,t\}$ and some easy calculations. Let $x$ be the class of $X$ in $S$. Then, $M:=Rt$ is the only maximal ideal of $R$, so that $(R,M)$ is a quasi-local ring. Moreover, $S=R[x]$, with $x\in S\setminus R$ satisfying $x^2\in M$ and $Mx\subseteq M$, so that $R\subset S$ is a minimal ramified extension \cite[Theorem 2.3]{DPP2}. It follows that the only maximal ideal of $S$ is $N:=Rx+Rt$, and we have the following relations: $t^2=x t=0$ and $x^2=t$, giving $N^2=Rx^2=Rt=M$ and $MN=Rt^2+Rtx=Rt^2=M^2=0$. Then, $R\subset S$ is a special minimal ramified extension. 

\begin{theorem}\label{2.17} Let $R\subseteq S_1,S_2$ be FIP extensions, $\Sigma_i:= \substack{+\\ S_i}R$ for $i=1,2$ and $\mathcal R:=S_1\times S_2$. Then $R\subseteq\mathcal{R}$ has FIP if and only if $ R$ is an FMIR such that $\mathrm{Supp}(\Sigma_1/R)\cap\mathrm{Supp}(\Sigma_2/R)\cap\Sigma\mathrm{Max}(R)=\emptyset$, and, for each $M\in\mathrm{Supp}(\Sigma_i/R)\cap\Sigma\mathrm{Max} (R),\ i\in\{1,2\}$, either $R_M\subset(\Sigma_i)_M$ is a special minimal ramified extension or $R_M$ is a field. 
\end{theorem}

\begin{proof} For a maximal ideal $M$ of $R$, we denote by $S(M)$ the seminormalization of $R_M$ in $(S_1\times S_2)_M$.

Assume that $R\subseteq S_1\times S_2$ has FIP. In view of Proposition~\ref{2.2},  $R$ is an FMIR, and so is a finite direct product $\prod_{i=1}^nR_i$ of fields, finite local rings and SPIRs  that are localization  of $R$ at some maximal ideal $M$ of $R$ by Proposition~\ref{2.3}. Hence $R_M\subseteq(S_1\times S_2)_M=(S_1)_M\times(S_2)_M$ has FIP by Proposition~\ref{2.7}.   Assume that $R_M$ is not a finite ring. Then, $R_M$ is either an infinite field or a $\Sigma$PIR. 

Let $M\in\Sigma\mathrm{Max}(R)$, so that $|R_M/M'|=\infty$ for $M':=MR_M$ (see the remark before Lemma~\ref{2.13}). For $j\in\{1,2\}$, we have that $R_M\subseteq(\Sigma_j)_M$ is a subintegral FIP extension with $(R_M,M')$ a quasi-local ring. Assume first that $R_M$ is a $\Sigma$PIR. Using Propositions~\ref{2.12}, ~\ref{2.15} and Corollary~\ref{2.16}, we get that $R_M=(\Sigma_j)_M$ for some $j\in\{1,2\}$, so that $M\not\in\mathrm{Supp}(\Sigma_j/R)$ and, for $l\in\{1,2\}\setminus\{j\}$, either $R_M =(\Sigma_l)_M$ or $R_M\subset(\Sigma_l)_M$ is a special minimal ramified extension. Assume now that $R_M$ is an infinite field. Using Proposition~\ref{2.12}, we get that $R_M=(\Sigma_j)_M$ for some $j\in\{1,2\}$ and, for $l\in\{1,2\}\setminus\{j\}$, there exists $\alpha\in(\Sigma_l)_M$ which satisfies $(\Sigma_ l)_M=R_M[\alpha]$ and $\alpha^3=0$ by \cite[Theorem 3.8]{ADM} since $R_M\subseteq(\Sigma_l)_M$ has FIP. Then, $M\not\in\mathrm{Supp}(\Sigma_1/R)\cap\mathrm{Supp}(\Sigma_2/R)$ and $\mathrm{Supp}(\Sigma_1/R)\cap\mathrm{Supp}(\Sigma_2/R)\cap\Sigma\mathrm{Max}(R)=\emptyset$. 

Conversely, assume that $R$ is an FMIR, and so a finite direct product $\prod_{i=1}^nR_i$ of fields, finite local rings and SPIRs such that $\mathrm{Supp}(\Sigma_1/R)\cap\mathrm{Supp}(\Sigma_2/R)\cap\Sigma\mathrm{Max}(R)=\emptyset$, with, for each $M\in\mathrm{Supp}(\Sigma_i/R)\cap\Sigma\mathrm {Max}(R),\ i\in\{1,2\}$, either $R_M\subset(\Sigma_i)_M$ is a special minimal ramified extension or $R_M$ is an infinite field. Observe first that for each $i$, there is $M\in\mathrm{Max}(R)$ such that $R_i=R_M$. 

Since $R$ is a quasi-semilocal ring, $\mathrm{MSupp}((S_1\times S_2)/R)$ is finite. Then, $R\subseteq S_1\times S_2$ has FIP if and only if $R_M\subseteq(S_1\times S_2)_M$ has FIP for each $M\in\mathrm{MSupp}((S_1\times S_2)/R)$ by Proposition~\ref{2.7}. Moreover, $R_M\subseteq(S_j)_M$ is an FIP extension for $j=1,2$. Fix $M\in\mathrm{MSupp}((S_1\times S_2)/R)$. Proposition~\ref{2.4} tells us that $\overline{\mathcal{R}}_M=(\overline{S_1})_M\times(\overline{S_2})_M=(\overline{S_1}\times\overline{S_2})_M\subseteq\mathcal{R}_M$ has FIP, where $\overline{\mathcal{R}}_M$ (resp$.$ $(\overline{S_i}) _M$) is the integral closure of $R_M$ in $({S_1})_M\times({S_2})_M=({S_1}\times{S_2})_ M$ (resp$.$ $({S_ i})_M$). Then, in view of \cite[Theorem 3.13]{DPP2}, $R_M\subseteq({S_1}\times{S_ 2})_M $ has FIP if and only if $R_M\subseteq(\overline{S_1}\times\overline{S_2})_M$ has FIP. From Proposition~\ref{2.10}, we deduce that $R_M\subseteq(S_1\times S_2)_M$ has FIP if and only if $R_M\subseteq S(M)$ has FIP. But, $S(M)=\substack{+\\ (\Sigma_1)_M\times(\Sigma_2)_M}R_M$ by Proposition~\ref{2.6}. Therefore, $S(M)$ is module finite over the Artinian ring $R_M$ by Proposition~\ref{2.8}. 

(1) If $R_M$ is an infinite field, then $M\in\Sigma\mathrm{Max}(R)$. We have $R_M=(\Sigma_l)_M$ for some $l\in\{1,2\}$ since $\mathrm{Supp}(\Sigma_1/R)\cap\mathrm{Supp}(\Sigma_2/R)\cap\Sigma\mathrm{Max}(R)=\emptyset$. Let $j\neq l$. Since $R_M\subseteq(\Sigma_j)_M$ has FIP, there is $\alpha_j\in(\Sigma_j)_M$ such that $(\Sigma_j)_M=R_M[\alpha_j]$, with $\alpha_j^3=0$ by \cite[Theorem 3.8]{ADM}. Moreover, $R_M[\alpha_j]$ is a quasi-local ring with maximal ideal $\alpha R_M[\alpha_j]$. Set $\alpha_l:=0$ and $\alpha:=(\alpha_1,\alpha_2)$. In view of Proposition~\ref{2.8}, we get $S(M)=R_ M[\alpha]$, with $\alpha^3=0$, so that $R_M\subseteq S(M)$ has FIP by \cite[Theorem 3.8]{ADM}. Indeed, $S(M)=R_M+(\alpha_jR_M [\alpha_j]\times 0)=R_M+\alpha R_M$.

(2) If $R_M$ is a $\Sigma$PIR, then $M\in\Sigma\mathrm{Max}(R)$, there is some $j\in\{1,2\}$ such that $(\Sigma_j)_M=R_M$, with, for $l\in\{1,2\}\setminus\{ j\}$, either $R_M=(\Sigma_l)_M$ or $R_M\subset (\Sigma_l)_M$ is a special minimal ramified extension. Then, $R_M\subseteq S(M)$ has FIP by either Proposition~\ref{2.15} or Corollary~\ref{2.5}.

(3) If $R_M$ is a finite ring, then $S(M)$ is a finite ring since a finitely generated $R_M$-module, and $R_M\subseteq S(M)$ has FIP.

In every case, $R_M\subseteq S(M)$ has FIP, and so has $R\subseteq R_1\times R_2$.
\end{proof} 

\begin{corollary}\label{2.18} Let $R\subseteq S_1,S_2$ be seminormal FIP extensions and $\mathcal R:=S_1\times S_2$. Then $R\subseteq\mathcal{R}$ has FIP if and only if $ R$ is an FMIR. 
\end{corollary}

\begin{proof} Since $R=\substack{+\\ S_i}R$ for $i=1,2$, we get $\mathrm{Supp}(\Sigma_1/R)\cap\mathrm{Supp}(\Sigma_2/R)\cap\Sigma\mathrm{Max}(R)=\emptyset$. Then, use Theorem~\ref{2.17}.\end{proof}

\section{FCP or FIP extensions and the CRT}

The aim of this section is to get an extension of the Chinese Remainder Theorem (CRT) in the following sense. Let $R$ be a ring, $n>1$ an integer and $I_1,\ldots, I_n$ ideals of $R$ distinct from $R$, but not necessarily distinct, such that $\cap_{j=1}^nI_j=0$. Such a family $\{I_1,\ldots,I_n\}$ of ideals of $R$ is called a {\it separating family}, a reference to Algebraic Geometry where a finite family of morphisms $\{f_j:M \to M_j \mid j= 1,\ldots , n\}$ of $R$-modules is called separating if $\cap_{j=1}^n\ker f_j= 0$. We intend to study the ring extension $R\subseteq \prod_{j=1}^n(R/I_j)=: \mathcal R$ associated to a separating family, denoting by $C:=(R: \mathcal R)$ its conductor, also called the {\it conductor of the separating family}. We set $J_j := (\cap _{k=1,k\neq j}^nI_k)$, or more generally $J_E: =\cap_{k=1, k\notin E}^n I_k$ for any subset $E$ of $\{1,\ldots,n\}$. We also denote by $e_i$ the element of $\mathcal{R}$ whose ith coordinate is $1$ and the others are $0$ and call $\{e_1,\ldots,e_n\}$ the ``canonical basis". The above extension is an isomorphism if $C=R$ (Chinese Remainder Theorem). If not, either $|[R,\mathcal R]|$ or $\ell [R,\mathcal R] $ measures in some sense how $R$ is far from to $\mathcal R$.

\begin{proposition}\label{3.1} Let $R$ be a ring and $\{I_1,\ldots,I_n\}$ a separating family of ideals of $R$. Then:

\begin{enumerate}
\item $R\subseteq \mathcal R$ is an infra-integral extension.

\item $C=\cap_{j=1}^n(I_j+J_j)=\sum_{j=1}^n J_j$.

\item $R\subseteq\mathcal R$ has FCP if and only if $R/ C$ is  Artinian.
\end{enumerate}
 \end{proposition}

\begin{proof} (1) Clearly, $R\to\prod_{j=1}^n(R/I_ j)$ is an integral ring extension (actually, module finite), that is infra-integral because of the form of elements  of $\mathrm{Spec}(\mathcal R)$(cf. the remark following Definition~\ref{1.3}).

(2) is \cite[Lemma 2.25]{V}.

(3) In view of \cite[Theorem 4.2]{DPP2}, we have that $R\subseteq \mathcal R$ has FCP if and only if $R/C$ is an Artinian. 
 \end{proof} 
 
An immediate consequence is the following. Let $R$ be a ring, $n>1$ an integer and $I_1,\ldots,I_n$ ideals of $R$ distinct from $R$, but not necessarily distinct. Set $C:=\sum_{j=1}^n J_j$. Then, $R/(\cap_{j=1}^nI_j)\subseteq\prod_{j=1}^n(R/I_j)$ has FCP if and only if $R/C$ is an Artinian ring. 

In the rest of the section, we examine the FIP property. The case of a separating family with two elements is easy to solve.

 \begin{proposition}\label{3.2} Let $R$ be a ring, with two ideals $I$ and $J$ such that $I, J\neq R$ and $I\cap J=0$. Then  $R\subseteq R/I\times R/J$ is a $\Delta_0$-extension, which  has FIP if and only if $R/(I+J)$ is an FMIR. 
 \end{proposition}

\begin{proof}  For $x\in R$, we denote by $\bar x$ its class in $R/I$ and by $\tilde x$ its class in $R/I$. 
 Set $e_1:=(\bar 1,\tilde 0)$, $e_2:=(\bar 0,\tilde 1)$, so that $\{e_1,e_2\}$ is a generating set of the $R$-module $R/I\times R/J$. From $e_i^2=e_i$ and $e_1e_2=0$ follow that $R/I\times R/J=R+Re_1$. Hence there is a bijection between the set of ideals of $R$ containing $ I+J$ and  $[R, R/I\times R/J]$ by Proposition~\ref{1.4} and  $R\subseteq R/I\times R/J$ has FIP if and only if $R/(I+J)$ is an FMIR. 
 \end{proof}

Next lemma  shows that we can reduce our study to a zero conductor extension. 

\begin{lemma}\label{3.3} Let $R$ be a ring and $\{I_1,\ldots,I_n\}$ a separating family of ideals of $R$. Then $R\subseteq \mathcal R$ has FIP if and only if the zero conductor extension $R/(\sum_{j=1}^nJ_j ) \subseteq\prod_{j=1}^n(R/(I_j+J_j))$ has FIP. 
 \end{lemma}

\begin{proof} By \cite[Proposition 3.7]{DPP2},  $R\subseteq \mathcal R$, with conductor $C$,  has FIP if and only if $R/C\subseteq \mathcal R/C$ has FIP. Since $C$ is an ideal of $\mathcal R $, for each $j\in\{1,\ldots,n\}$, there exists an ideal $C_j$ of $R$ containing $I_j$ such that $C= \prod_{j=1}^n C_j/I_j$. Now, there is a natural isomorphism $\mathcal R/C\cong \prod_{j=1}^n(R/C_j)$. For each  $j$, we get that $C_{j}/I_{j}=(I_{j}+ J_j)/I_j$ because $I_{j}+\sum_{i=1}^n J_i= J_j+(\sum_{i=1,i\neq j}^n J_i)+I_{j}=I_{j}+ J_j$.  Then, $R/C_{j}\cong(R/I_{j})/(C_{j}/I_{j})\cong(R/I_{j})/((I_{j}+J_j)/I_{j})\cong R/(I_{j}+J_j)$ giving the wanted result.
\end{proof} 

\begin{proposition}\label{3.4} Let $R$ be a ring and $\{I_1,\ldots,I_n\}$ a separating family of ideals of $R$ with zero conductor. Then: 

\begin{enumerate}
\item $J_j=0$ for each $j$.

\item If $R\subseteq \mathcal R$ has FIP, then $R/(J_{ \mathcal{P}_1}+ J_{ \mathcal{P}_2})$ is an FMIR for any partition $\{\mathcal{P}_1,\mathcal{P}_2\}$ of $\{1,\ldots, n\}$ as well as $R/I_j$  for each $j$. In that case, $R$ is  an Artinian ring.
\end{enumerate}
 \end{proposition}

\begin{proof} (1) By Proposition~\ref{3.1},   $C=\sum_{j=1}^n J_j$, so that $J_j= 0$.

(2) Set $K_i:=J_{\mathcal{P}_i}$ for $i=1,2$. Then, $K_1\cap K_2=0$, so that we have the extensions $R\subseteq R/K_1\times R/K_2$ and $R/K_i\subseteq\prod_{j\in\mathcal{P}_l}(R/I_j)$ for $l\neq i,\ l\in\{1,2 \}$ leading to the composite $R\subseteq R/K_1\times R/K_2\subseteq\mathcal R$. If $R\subseteq\mathcal R$ has FIP, then so has $R\subseteq R/K_1\times R/K_2$. By Proposition~\ref{3.2}, $R/(K_1+K _2)$ is an FMIR. The second statement follows from (2) and $J_j= 0$. To complete the proof, use Proposition~\ref{3.1} since $C= 0$.
\end{proof} 

The following result shows that the case of a nonlocal Artinian ring $R$ is very different from the local case.

\begin{proposition}\label{3.5} Let $R$ be a ring containing a set of $p> 1$  orthogonal idempotents $\{e_1,\ldots,e_p\}$, generating the  ideal $R$. Then $R$ is an FMIR if $R \subseteq \mathcal{R}$ has FIP for each separating family $\{I_1,\ldots, I_n\}$ of ideals of $R$. In particular, an Artinian nonlocal ring $R$ is an FMIR if  $R \subseteq \mathcal {R}$ has FIP for each separating family  of ideals of $R$. The converse holds if no local ring of $R$ is a SPIR.
\end{proposition}
\begin{proof}  Consider the faithfully flat extension $R \subseteq \prod_{i=1}^p R/Re_i=: S$  with zero conductor (Proposition~\ref{3.1}). If $R \subseteq S$ has FIP, then each $R/Re_i$ is an FMIR by Proposition~\ref{3.4} and so is $S$. Then observe that if $R \to S$ is a faithfully flat ring morphism, $R$ is an FMIR if $S$ is, because $IS \cap R= I$ for each ideal $I$ of $R$. Now if $R$ is Artinian nonlocal, then $R$ has $p> 1$ idempotents generating the ideal $R$ by the Structure Theorem of Artinian rings.
\end{proof}

Now let $(R,M)$ be a local Artinian ring with $|R/M|<\infty$. Then $|R|<\infty$ (see the remark before Lemma~\ref{2.13}), so that $R\subseteq \mathcal R$ has $\mathrm{FIP}$ for each separating family, since $|\mathcal R|<\infty$. 

We know that $|\mathrm{MSupp}(S/R)|<\infty$ if $R\subseteq S$ has FIP (Proposition~\ref{2.7}(1)). By  Proposition~\ref{2.7} and former results of the section,  the FIP property study can be reduced to the next proposition hypotheses.

If $(R,M)$ is an Artinian local ring, we denote by $n(R)$ the nilpotency index of $M$.

\begin{proposition}\label{3.6} Let $(R,M)$ be an Artinian local ring with $|R/M|=\infty$  and  a separating family $\{I_1,\ldots,I_n\}$ of ideals, with $C=0$.

\noindent We set $ T:= R+M\mathcal R,\ \mathcal{C}:=(R:T),\ n(R/\mathcal{C}) =p$, and for each $i>0,\ M_i:=M+TM^{i}=M+\mathcal RM^{i+1}$, $R_i:=R+TM^{i}=R+\mathcal RM^{i+1}$. Then, 
\begin{enumerate}

\item $T=\substack{+\\\mathcal R}R$ and $R\subseteq \mathcal R$ has $\mathrm{FIP}$ if and only if $R\subseteq T$ has FIP. 

\item $\mathcal{C}=(0:M)$.

\item $R\subseteq T$ has FIP if and only if either $R=T$, or $R_1=T$, or $R_1\subset T$ is minimal (ramified), with, in the two last situations, either $M=(R:T)$, or ${\mathrm L}_R(M_i/M_{i+1})=1$ for all $1\leq i\leq p-1$.
\end{enumerate}

The case $R=T$ corresponds to an extension of the form $K\subseteq K^n$, where $K$ is a field, and the case $M=\mathcal{C}$ to $M^2=0$.
 \end{proposition}

\begin{proof} Let $\{e_1,\ldots,e_n\}$ be the canonical basis of the $R$-module $\mathcal{R}$. 
Since  $(R:\mathcal R) = 0$,  $J_j=0$ for each $j\in\{1, \ldots,n\}$ by Proposition~\ref{3.4}. 

(1) $T= \substack{+\\\mathcal R}R$ follows from \cite[Theorem 5.18]{DPP2} since $\mathrm{Rad}(\mathcal R)=M\mathcal R$ and $R\subseteq \mathcal{R}$ has FCP by Proposition~\ref{3.1}. Since $R\subseteq\mathcal R$ is infra-integral, $R\subseteq \mathcal R$ has $\mathrm{FIP}$ if and only if $R\subseteq T$ has FIP by Proposition~\ref{1.6}. 

(2) is an easy calculation, because each $J_j = 0$, $\cap_{j=1}^n I_j = 0$ and  the unit element of $\mathcal R$ is $e_1+\cdots +e_n$.

(3) Since $R\subseteq R_i\subseteq T$ is finite and subintegral, $(R_i,M_i)$ is local Artinian for each $i>0$. We have $TM=M+\mathcal RM^2=M_1\subseteq\mathcal RM \in \mathrm{Max}(T)$, $R_1 =R+\mathcal RM^2,\ R_2=R+\mathcal RM^3$ and $M_2=M+\mathcal RM^3$. Because $R/M$ is infinite, \cite[Theorem 5.18]{DPP2}, applied with $S:=\mathcal R $, gives that $R\subseteq T$ has FIP if and only if the next two properties hold:

(i) Either $R=T$, or $M=(R:T)$, or ${\mathrm  L}_R(M_i/M_{i+1})=1$ for all $1\leq i\leq p-1$;
 
(ii) If $R\neq T$, there exists $\alpha\in T$ such that $T=R_1[\alpha]$ and $\alpha^3\in TM$, and, with $T':=R_1[\alpha^2]$ and $T'':=R+T'M$, there exists $\beta\in T$ such that $T'=T''[\beta]$ and $ \beta^3\in T'M$.
 
Assume that $T\neq R,R_1$, so that $\alpha\not\in R_1$. We first show that (ii) implies that $R_1\subset T$ is minimal. Let $\alpha\in T$ such that $\alpha^3\in TM=M_1\subseteq\mathcal RM$, giving $\alpha\in\mathcal RM$, so that $\alpha^2\in\mathcal RM^2\subseteq M_1$ and $\alpha M_1\subseteq\mathcal R MM_1=\mathcal RM(M+\mathcal RM^2)\subseteq\mathcal RM^2\subseteq M_1$. Then, $R_1\subset T$ is minimal (ramified) in view of \cite[Theorem 2.3(c)]{DPP2}.

Conversely, we show that $R_1\subset T$ is minimal (ramified), with either $M=(R:T)$, or ${\mathrm L}_ R(M_i/M_{i+1})=1$ for all $1\leq i\leq p-1$ implies (ii). Actually, (i) already holds. Since $R_1\subset T$ is minimal, there is $\alpha\in T$ such that $T=R_1[\alpha]$ and $\alpha^2\in M_1\subset R_1$, with $\alpha M_1\subseteq M_1$. Then, $\alpha^3\in M_1=TM$. Now, we can rewrite (ii) as $T'=R_1[\alpha^2] =R_1$ and $T''=R+T'M=R+R_1M=R+\mathcal RM^3=R_2$. Assume that $M\neq(R:T)=(0:M)$, so that $ M^2\neq 0$. Then, $M_1^2=(M+\mathcal RM^2)^2\subseteq M+\mathcal RM^3=M_2\subset M_1$ (because ${\mathrm L}_R(M_1/M_2)=1$) implies that $R_2\subset R_1$ is minimal ramified by Theorem~\ref{1.2}(c). Arguing as for $\alpha$, we obtain some $\beta\in T$ such that $T'=T''[\beta]$ and $ \beta^3\in T'M$ and (ii) holds.

If $T=R_1$, it is enough to take $\alpha=0$ to get (ii). 

If $R=T$, then $I_j=M$ for each $j$ entails  $M= \cap_{j=1}^nI_j=0$ and $R$ is a field.  Then  $R\subseteq \mathcal{R}$ is  of the form $K\subseteq K^n$, where $K$ is a field, and has  FIP  (see Proposition~\ref{2.1}). Assume that $M=\mathcal{C}$, then $M^2=0$.
\end{proof} 

By Proposition~\ref{3.4}, we know that when $R\subseteq\mathcal{R}$ has FIP, then $R/I_j$ is an FMIR for each $j$. It is natural to ask if the converse holds, and if not, what conditions are needed to get the FIP property. We consider here a simple case which already gives  a rather complicated result.

\begin{proposition}\label{3.6 bis} Let $(R,M)$ be an Artinian local ring such that $M^2= 0$ and $|R/M|=\infty$. Let $\{I_1,\ldots,I_n\}$ be a separating family of ideals, with conductor $0$ and $n\geq 3$. Then, $ R\subseteq\mathcal{R}$ has FIP if and only if $R/I_j$ is an FMIR and $M=I_j+\cap_{k\neq j,l}I_k$, for each $j,l\in\{1,\ldots,n\},\ j\neq l$.  
 \end{proposition}

\begin{proof} Set $T:=R+M\mathcal R,\ \mathcal{C}:=(R:T)$, and for each $i>0,\ M_i:=M+TM^{i}=M+\mathcal RM^{i+1},\ R_i:=R+TM^{i}=R+\mathcal RM^{i+1}$. Since $M^2=0$, we get that $R_1=R$ and $ M_1=M=M_2$. Then, applying Proposition~\ref{3.6}, we have that $R\subseteq\mathcal R$ has FIP if and only if $R\subseteq T$ has FIP, if and only if either $R=R_1=T$, or $R\subset T$ is minimal (ramified), with $M=(R:T)$. This last condition is always satisfied since $\mathcal{C}=(0:M)$. Then, $R\subseteq \mathcal R$ has $\mathrm{FIP}$ if and only if  either $R=R_1=T$, or $R\subset T$ is minimal. 

We begin to remark that $M=I_k$ for at least $n-1$ ideals $I_k$ implies that $M=0$, so that $R$ is a field and we are in the situation of Proposition~\ref{2.1}. Indeed, if $n-1$ ideals $I_k$ are equal to $M$, for instance $I_1,\ldots,I_{n-1}$, we get that $\cap_{k\neq n}I_k=M=0$ since $(R:\mathcal R)=0$. In particular, we get that $I_n=0$. Hence, the assertion of Proposition~\ref{3.6 bis} holds.

So, in the following, we may assume that there exist some $I_j,I_l\neq M,\ j\neq l$. Consider the following $R$-subextension of $(R/I_j)\times R$ defined by $R'_j:=R+((M/I_j)\times 0)=\{(\overline x+\overline m,x)\mid x\in R,\ m\in M\}$. Since $\cap_{k\neq j}I_k=0$, we have the ring extension $R\subseteq R+\prod_{k\neq j}M/I_k$. An easy calculation shows that we have a ring extension $R'_j\subseteq T$. Moreover, $ R\neq R'_j$ since $(\overline m,0)\in R'_j\setminus R$ for any $m\in M\setminus I_j$. In particular, $R\neq T$. The canonical map $\varphi:R'_j\to T$ is defined by $\varphi(\overline x+\overline m,x)=(\overline x,\ldots,\overline x)+(\overline m,\overline 0,\ldots,\overline 0)$ (after reindexing the components).

Assume first that $R\subseteq\mathcal R$ has FIP, so that $R\subset T$ is a minimal extension. Then, $ R\neq R'_j$ implies that $R'_j=T$ and $\varphi$ is surjective. Let $y\in M$ and $j'\in\{1,\ldots,n\},\ j'\neq j$. Consider $(\overline 0,\ldots,\overline y,\ldots,\overline 0)\in T$, where all the coordinates are $\overline 0$ except possibly the $j'$th which is $\overline y$. Then, there exist $x\in R,\ m\in M$ such that $(\overline 0,\ldots,\overline y,\ldots,\overline 0)=(\overline x,\ldots,\overline x)+(\overline m,\overline 0, \ldots,\overline 0)$. This gives $y-x\in I_{j'},\ x+m\in I_j$ and $x\in I_k$ for each $k\neq j,j'$. Then, $x\in\cap_{k\neq j,j'}I_k$ and $y\in I_{j'}+\cap_{k\neq j,j'}I_k$, giving $M=I_{j'}+\cap_{k\neq j,j'}I_k$ for any $j' \neq j$. Since there is some $l\neq j$ such that $M\neq I_l$, the same reasoning gives that $M=I_j+\cap_ {k\neq j,l}I_k$. At last, if there exist $j',l'\in\{1,\ldots,n\},\ j'\neq l'$ such that $M\neq I_{j'},I_{l'}$, the same reasoning gives again $M=I_{j'}+\cap_{k\neq j',l'}I_k$. But, when $M=I_{j'}$, we have $M=I_{j'}+\cap_{k\neq j',l'}I_k$ whatever is $I_{l'}$.

Conversely, assume that $R/I_{j'}$ is a FMIR and $M=I_{j'}+\cap_{k\neq j',l'}I_k$, for each $j',l'\in\{1,\ldots, n\},\ j'\neq l'$, with $M\neq I_j$ for some $j$. We are going to show that $R\subset R'_j$ is minimal ramified and that $R'_j=T$.

Since $R/I_j$ is an FMIR with $|R/M|=\infty$ and $M\neq I_j$, there exists some $z\in M\setminus I_j$ such that $M/I_j=(R/I_j)\overline z$, with $\overline z\neq 0$ and $\overline z^2=0$. Set $t:=(\overline z,0)\in R'_j\setminus R$. Using the properties of $R'_j$, we get that $R'_j=R[t]$, with $t^2=0\in M,\ tM=0\subseteq M$, so that $R\subset R'_j$ is a minimal ramified extension by \cite[Theorem 2.3]{DPP2}. 

Let $j'\neq j$ and $x\in M$. Since $M=I_{j'}+\cap_{k\neq j',j}I_k$, there exist $a\in I_{j'}$ and $b\in\cap_{k\neq j',j}I_k$ such that $x=a+b$. Then, $\overline x=\overline b$ in $M/I_{j'}$. It follows that we get $(\overline 0,\ldots,\overline x,\ldots,\overline 0)=(\overline b,\ldots,\overline b,\ldots,\overline b)+(\overline 0,\ldots,-\overline b,\ldots,\overline 0)$, where $\overline x$ stands at the $j'$th component in the first element, and $-\overline b$ stands at the $j$th component in the last element. Indeed, for $k\neq j,j'$, we have $\overline b=\overline 0$ since $b\in\cap_{k\neq j',j}I_k$. We have $(\overline b,\ldots,\overline b,\ldots,\overline b)\in R$ and $(\overline 0,\ldots,-\overline b,\ldots,\overline 0) \in(M/I_j)\times 0$, so that $(\overline 0,\ldots,\overline x,\ldots,\overline 0)\in R'_j$. This holds for any $j' \neq j$ and obviously for $(\overline 0,\ldots,\overline x,\ldots,\overline 0)$ where $\overline x$ stands at the $j$th component by definition of $R'_j$. Then, $T=R+\prod_k(M/I_k)=R+((M/I_j)\times 0)=R'_j$, giving that $R\subset T$ is minimal, so that $R\subset \mathcal{R}$ has FIP. 
\end{proof} 

\begin{remark}\label{3.6ter} When $n=3$, the condition of Proposition~\ref{3.6 bis} becomes $M=I_j+I_l$, for each $j,l\in\{1,2,3 \},\ j\neq l$. Here is an example where $I_j\not\subseteq I_l$ for each $j,l\in\{1,2,3 \},\ j\neq l$.

Let $k$ be an infinite field, and set $R:=k[X,Y]/(X,Y)^2=k[x,y]$, for some indeterminates $X,Y$. Then, $R$ is an Artinian local ring with maximal ideal $M:=(x,y)$ such that $M^2= 0$ and $|R/M|=\infty$. Set $I_j:= k(x+\lambda_jy)$, where $\lambda_1,\lambda_2$ and $\lambda_3$ are three distinct elements of $k$. Then, $I_j\cap I_l=0$ for each $j,l\in\{1,2,3 \},\ j\neq l$. We have $R/I_j=k[\overline x]$, which is a SPIR, although $R$ is not a SPIR, with $M/I_j=k\overline x$.  
\end{remark}

In the following, we are going to consider a kind of converse for Proposition~\ref{3.4}, taking for $R$ a local FMIR. By Proposition~\ref{2.3}, either $R$ is a field, or a finite ring, or a $\Sigma$PIR. The case where $R$ is a field is Proposition~\ref{2.1}. If $R$ is a finite ring, $\mathcal{R}$ being  $R$-module finite, $\mathcal{R}$ is also a finite ring, so that $R\subseteq\mathcal{R}$ has FIP. The last case to consider  is a $\Sigma$PIR $R$. 

\begin{proposition}\label{3.7} Let $(R,M)$ be a $\Sigma$PIR and a separating family  $\{I_1, \ldots,I_n\}$ of ideals,  with  conductor $0$. Then, $R\subseteq\mathcal{R}$ has FIP if and only if either $n=2$,  or $I_j=M$ for $n-2$ ideals $I_j$. 
 \end{proposition}

\begin{proof} For $n=2$, we get $I_1=I_2=0$ and Corollary~\ref{2.5} gives that $R\subseteq R/I_1\times R/I_2$ has FIP.

Assume that $n>2$. The ideals of the SPIR $R$  are linearly ordered. Thus  we can assume $I_1 \subseteq\cdots\subseteq I_j\subseteq\cdots\subseteq I_n$. By Proposition~\ref{3.4}, we get that $J_j=0$ for each $j\in\{1,\ldots,n\}$. Hence, for $j=1$, we get $I_2= 0$ and $I_1= 0$ for $j\neq 1$. Moreover, there is some $t\in M$ such that $M=Rt$, with $t^p=0,\ t^{p-1}\neq 0$ for some positive integer $p>1$ since $R$ is not a field, and, for each $j\in\{1,\ldots,n\}$, there is an integer $p_j> 0$ such that $I_j=Rt^{p _j}$, with $I_ j\neq Rt^{p_j-1}$. In particular, we have $p=p_1=p_2\geq\cdots\geq p_j\geq\cdots\geq p_n$. 

Assume that $I_3\neq M$, whence $p_3>1$. Let $\{e_1,\ldots,e_n\}$ be the canonical basis of $\mathcal {R}$ over $R$ and $\mathcal{F}$ a set of representative elements of $R/M$. For each $\alpha\in\mathcal {F}$, set $R_{\alpha}:=R+R(t^{p-1}e_2+\alpha t^{p_3-1}e_3)$, which is an $R$-subalgebra of $\mathcal {R}$. Let $\alpha,\beta\in\mathcal{F},\ \alpha\neq\beta$, so that $\alpha-\beta\not\in M$. Assume that $R_ {\alpha}=R_{\beta}$. Then, $t^{p-1}e_2+\alpha t^{p_3-1}e_3\in R_{\beta}$, so that there exist $a,b\in R$ such that $t^{p-1}e_2+\alpha t^{p_3-1}e_3=a\sum_{j=1}^ne_j+b(t^{p-1}e_2+\beta t^{p_3-1}e_3)$. This gives $a=0,\ t^{p-1}(1-b)=0\ (*)$ and $t^{p_3-1}(\alpha-b\beta)\in I_3\ (**)$.  But we get $1-b\in M$ by $(*)$ and $\alpha-b\beta\in M$ by $(**)$, so that $\alpha-\beta\in M$, a contradiction; whence $R_{\alpha}\neq R_{\beta}$, and $R\subseteq\mathcal{R}$ has not FIP by Lemma~\ref{2.13}.

Now, assume that $n>2$ and $I_j=M$ for all $j\geq 3$. Using the notation of Proposition~\ref{3.6}, we get that $T=R+(M\times M)\subseteq R^2$. But $R\subseteq R^2$ has FIP by Corollary~\ref{2.5}, so that $R\subseteq T$ has FIP, inducing that $R\subseteq\mathcal{R}$ has FIP by Proposition~\ref{3.6}.
\end{proof}

\begin{corollary}\label{3.8} Let $(R,M)$ be a quasi-local ring such that $|R/M|=\infty$, and a separating family $\{I_1,\ldots,I_n\}$ of ideals of $R$. Assume that $R/(\sum_{i=1}^n J_j)$ is a SPIR. Then, $R\subseteq \mathcal{R}$ has FIP if and only if either $n=2$,  or $I_j+J_j=M$ for $n-2$ ideals $I_j+J_j$. 
 \end{corollary}
 
 \begin{proof} Set $R':=R/(\sum_{i=1}^n J_j)=R/C$, where $C:=(R:\mathcal R)$, so that $R\subseteq\mathcal R$ has FIP if and only if $R'\subseteq\prod_{j=1}^n(R/(I_j+J_j))$ has FIP (Lemma~\ref{3.3}). Then, apply Proposition~\ref{3.7} to this extension.
 \end{proof} 

\begin{remark}\label{3.9} Let $(R,M)$ be a local Artinian ring such that $|R/M|=\infty$, and a separating family $I_1,\ldots,I_n$ of ideals of $R$ different from $M$, with $n>2$,  associated extension $R\subseteq\mathcal{R}$ and conductor $C$. We give below such an extension having  FIP while $R/C$ is not an FMIR. 

Let $K$ be an infinite field,  $R:=K[X,Y]/(X,Y)^2$ with maximal ideal $M$. Then $(R,M)$ is a local Artinian ring with $M^2=0$ and  $R/M\cong K$  infinite. Let $x,y$ be the classes of $X,Y$ in $R$, $I_1:=Rx,\ I_2:=Ry,\ I_3:=R(x+y)$ and $\mathcal R:=\prod_{j=1}^3(R/I_j)$. From $I_j\cap I_k=0$ for each $j\neq k\in\{1,2,3\}$, we deduce that $C=0$ by Proposition~\ref{3.1} and also that $\{I_1,I_2,I_3\}$ is a separating family. Let $\overline a$ be the class of $a\in R$ in any $R/I_j$. Observe that $M/I_1=(R/I_1)\overline{y},\ M/I_2=(R/I_2)\overline{x},\ M/I_3=(R/I_3)\overline{x}$, because $y=(x+y)-x$. Hence each $M/I_j$ is a principal ideal with $(M/I_j)^2=0$, so that each $R/I_j$ is a SPIR. Set $e_1:=(\overline{y},\overline{0},\overline{0}),\ \alpha:=e_2:= (\overline{0},\overline{x},\overline{0}),\ e_3:=(\overline{0},\overline{0},\overline{x})$. Using the notation of Proposition~\ref{3.6}, we have $(R:T)=M,\ T=R+\mathcal RM=R+\sum_{i=1}^3Re_j$ and $R_1=R+\mathcal R M^2=R$. Since $(\overline{0},\overline{x},\overline{x})=x\in R$, we get $e_2+e_3=x$, whence $e_3=x-\alpha$. At last, $e_1= (\overline{y},\overline{0},\overline{0})=(\overline{x+y},\overline{0},\overline{0})=(\overline{x+y},\overline{x+y},\overline{x+y})-(\overline{0},\overline{x},\overline{0})=(x+y)-\alpha$. It follows that $T=R [\alpha]$, with $\alpha^2=0$ and $M\alpha=0$, so that $R=R_1\subset T$ is a minimal ramified extension \cite[Theorem 2.3]{DPP2}. Then,   $R\subset T$ and $R=R_1\subset\mathcal R$ have FIP  by Proposition~\ref{3.6}, although $(R,M)$ is a local ring which is not a SPIR: the set of ideals $\{R(x+ay)\mid a\in\mathcal F\}$ is infinite, if $\mathcal F$ is a set of representative elements of $R/M\cong K$.  
\end{remark}

\begin{corollary}\label{3.10} Let $(R,M)$ be a quasi-local ring with $|R/M|=\infty$. Let $I, J$ be ideals of $R$ with $I\cap J=0$ and such that $S:= R/(I+J)$ is a SPIR with nilpotency index $n(S) =p>0$ if $I+J \neq R$. 

\begin{enumerate}
\item Assume that $I+J= R$. Then, $|[R,R/I\times R/J]|=1$.

\item Assume that $I+J\neq R$. Then, $|[R,R/I\times R/J]|=p+1$.
\end{enumerate}

In particular, if $(R,M)$ is a SPIR with $n(R)=q\geq 1$, then $|[R,R^2]|=q+1$ and $[R,R^2]=\{R+M^{i}R^2\}_{i=0,\ldots,q}$.
\end{corollary}

\begin{proof} (1) If $I+J=R$,  then $|[R,R/I\times R/J]|=1$ by the CRT. 

(2) Assume now that $I+J\neq R$. Since $(S,N)$ is a SPIR with $N:=M/(I+J)$,  $R\subseteq R/I\times R/J$ has FIP by Proposition~\ref{3.2} and its conductor $C= I+J$ by Proposition~\ref{3.1}. Moreover, the proof of Proposition~\ref{3.2} shows that there is a bijection between $[R,R/I\times R/J]$ and the set of ideals of $R/C= S$. Since $(S,N)$ is a SPIR, there is some $t\in S$ such that $N=St$ and the ideals of $S$ are linearly ordered. Then, this set of ideals is $\{St^k\mid k\in\{0,\ldots,p\}\}$ and $|[R,R/I\times R/J]|=p+1$.

Now if $(R,M)$ is a SPIR, with  $n(R) =q\geq 1$, we deduce from (2) that  $|[R,R^2]|=q+1$. Since $(R,M)$ is a SPIR, there exists $x\in R$ such that $M=Rx$ and the ideals of $R$ are the $Rx^ {i}$, for $i=0,\ldots,q$. Moreover, the bijection $\varphi$ between the set of ideals of $R$ and $[R,R^2]$ is given by $\varphi(Rx^{i})=R+x^{i}R^2$. 
 \end{proof}

We  next generalize  some Ferrand-Olivier's result \cite[Lemme 1.5]{FO}.

\begin{theorem}\label{3.11} Let $R$ be a ring, $\{I_1,\ldots,I_n\}$, $n>2$,  a separating family of ideals of $R$. Then, $R\subseteq\mathcal{R}$ is a minimal extension if and only if the following condition $(\dag)$ holds:

$(\dag)$:  There exist $j_0,k_0\in\{1,\ldots,n\},\ j_0\neq k_0$ such that $I_{j_0}+I_{k_0}\in\mathrm{Max}(R)$ and $I_j+I_k=R$ for any $(j,k)\neq (j_0,k_0),\ j\neq k$.

If $(\dag)$ holds, then $\{I_1,\ldots, I_n\}$ satisfies a weak Chinese Remainder Theorem: $I_j+\cap_{k\neq j}I_k=\cap_{k\neq j}(I_j+I_k)$ for each $j\in\{1,\ldots,n\}$.  
\end{theorem}

\begin{proof} Assume first that $(\dag)$ holds.  There is no harm to suppose that $j_0=1$, $k_0=2$ and set $J:=\cap_{j=2}^nI_j$. Then $I_j+I_k=R$ for any $j,k\geq 2,\ j\neq k$ gives that $\prod_{j=2}^n(R/I_j)\cong R/J$. So, we are reduced to  the extension $R\subseteq R/I_1\times R/J$. But, $I_1+I_j=R$ for each $j>2$ and $I_1+I_2=M$ give $I_1+J=M$ because $I_1+J\subseteq M$. For the reverse inclusion,  consider in $R/I_1$ the relations $\overline 1=\overline x_j\ (*_j)$ for some $x_j\in I_j$, for any $j>2$. Let $m\in M$. There is $x_2\in I_2$ with $\overline m=\overline x_2$ in $R/I_1$. Using $(*_j)$, we get that $\overline m=\overline x_2\cdots\overline x_n$, so that $ m\in I_1+J$. Then, by \cite[Lemme 1.5]{FO},  $R\subseteq\mathcal{R}$ is a minimal extension since $I_1\cap J=0$.

Conversely, if $R\subseteq\mathcal{R}$ is minimal (integral), then $M:=(R:\mathcal{R})\in\mathrm{Max} (R)$ is an ideal of $\mathcal{R}$. Moreover, there is some $N_1\in\mathrm{Max}(\mathcal{R})$ above $ M$ and possibly only another one $N_2$. There is no harm to suppose that $N_1=M/I_1\times \prod_{k= 2}^nR/I_k$ with $I_1\subseteq M$ and $N_2=R/I_1\times M/I_2\times\prod_{k=3}^n R/I_k$ with $I_2 \subseteq M$. Any other $M'\neq M$ in $\mathrm{Max}(R)$, is lain over by a unique element of $\mathrm {Max}(\mathcal{R})$, of the form $M'S=\prod_{j=1}^n((M'+I _j)/I_j)$ by \cite[Lemma 2.4]{DPP2}. Then, $M'+I_j=R$ for all $j$ but one, so that there is a unique $I_j$ contained in $M'$. Then, for any $j,k>2,\ j\neq k$ and $i=1,2$, we have $I_j+I_k=I_i+I_j=R$, which gives $\prod_{j=2}^n(R/I_j)\cong R/J$ where $ J:=\cap_{j=2}^nI_j$. So, the minimal extension $R\subseteq R/I_1\times R/J$ is involved. By \cite[Lemme 1.5]{FO}, we get that $I_1+J=M''$, for some $M''\in\mathrm{Max}(R)$, whence $I_1,J\subseteq M''$.  Actually, we have $M= M''$. Deny, then $I_j\not\subseteq M''$ for all $j\geq 2$ gives $J\not\subseteq M'' $, a contradiction. A similar proof gives $I_2\subseteq M$ since $J\subseteq M$. From $M=I_1+J\subseteq I_1+I_2\subseteq M$,  we get that  $I_1+I_2=M$ and the proof is complete.

Assume that $(\dag)$ holds,  then easy calculations show that $I_ j+\cap_{k=1,k\neq j}^nI_k$ $=\cap_{k=1,k\neq j}^n(I_j+I_k)$ for each $j\in\{1,\ldots,n\}$, so that $\{I_1,\ldots, I_n\}$ satisfies a weak  Chinese Remainder Theorem. 
 \end{proof}

\section{The case of ring powers}

In this section, we consider   separating families whose ideals are zero. 

\begin{proposition}\label{4.1} Let $(R,M)$ be a $\Sigma$PIR and an integer $n>1$.  Then $R\subseteq R^n$  has FIP if and only if $n=2$.
\end{proposition}
\begin{proof}  Use Proposition~\ref{3.7} with $I_j=0$ for each $j$. Since $(R:R^n)=0$ and $M\neq 0$, we get the result. 
\end{proof}

We are now in position to get a result in the general case.

\begin{theorem}\label{4.2} Let $R$ be a ring and $n>1$ an  integer. Then $R\subseteq R ^n$ has FIP if and only if $R$ is an FMIR with $n=2$ when $R$ is a $\Sigma$FMIR.

 \end{theorem}

\begin{proof} For any maximal ideal $M$ of $R$, we have $(R^n)_M=(R_M)^n\neq R_M$, so that $\mathrm {MSupp}(R^n/R)=\mathrm {Max}(R)$. 

Assume that $R\subseteq R^n$ has FIP. Using Proposition~\ref{3.4} with $I_j=0$ for each $j$ and since $(R:R^n)=0$, we get that $R$ is an FMIR. Moreover, $R_M\subseteq(R_M)^n$ has FIP for each $M\in\mathrm{Max}(R)$ in view of Proposition~\ref{2.7}. Assume that there is some $M\in\mathrm{Max}(R)$ such that $R_M$ is a $\Sigma$PIR.  Since $MR_M\neq 0$, we get that $n\leq 2$ by Proposition~\ref{4.1}, so that $n=2$. 

Conversely, if $R$ is an FMIR, then $|\mathrm {Max}(R)|<\infty$ and $R\subseteq R^n$ has FIP if and only if $R_M\subseteq(R_M)^n$ has FIP for each $M\in\mathrm{Max}(R)$. Let $M\in\mathrm{Max}(R)$. If $R_M$ is a field, then $R_M\subseteq(R_M)^n$ has FIP by Proposition~\ref{2.1}. If $R_M$ is a finite ring, then so is $(R_M)^n$ and $R_M\subseteq(R_M)^n$ has FIP. Assume that $R_M$ is a $\Sigma$PIR, so that $R$ is a $\Sigma$FMIR and $n=2$. Then, Proposition~\ref{4.1} gives that $R_M\subseteq (R_M)^n$ has FIP. Therefore, $R\subseteq R^n$ has FIP.
 \end{proof}

We get now a generalization of Theorem~\ref{2.17}.

\begin{theorem}\label{4.3} Let $R\subseteq S_j,\ j=1,\ldots,n$ be finitely many FIP extensions, $\Sigma_ j:=\substack{+\\ S_j}R$ and $S:=\prod_{j=1}^nS_j$. Then $R\subseteq S$ has FIP if and only if $R$ is an FMIR satisfying the following conditions $(B_1)$ and $(B_2)$:

\noindent $(B_1)\ \mathrm{Supp}(\Sigma_j/R)\cap\mathrm{Supp}(\Sigma_j/R)\cap\Sigma\mathrm{Max}(R)=\emptyset$ for any $j,l\in\{1,\ldots,n\}$ such that $ j\neq l$.

\noindent $(B_2)$ If there exists $M\in\Sigma\mathrm{Max}(R)$ such that $R_M$ is a $\Sigma$PIR, then $n=2$ and, for each such $M$ and each $j\in\{1,2\}$, either $R_M\subset(\Sigma_j)_M$ is a special minimal ramified extension or $R_M=(\Sigma_j)_M$. 
\end{theorem}

\begin{proof}  The result can be written under the form (A) $\Leftrightarrow$ $R$ is an FMIR satisfying  conditions $(B_1)$ and $(B_2)$  where (A) is the statement: $R\subseteq S$ has FIP. 

Assume that (A) holds. Then, $R\subseteq R^n$ has FIP. In view of Theorem~\ref{4.2}, $R$ is an FMIR and $n=2$  as soon as $R$ is a $\Sigma$FMIR, in which case  we can use Theorem~\ref{2.17}.

If there exist $j,l\in\{1,\ldots,n\},\ j\neq l$ and $M\in\mathrm{Supp}(\Sigma_j/R)\cap\mathrm{Supp}(\Sigma _l/R)\cap\Sigma\mathrm{Max}(R)$, then $R_M\neq(\Sigma_j)_M,(\Sigma_l)_M$, with $R_M$ infinite. Moreover, $R_M\subset(\Sigma_ j)_M$ and $R_M\subset(\Sigma_l)_M$ are subintegral extensions. In view of Proposition~\ref{2.12}, we get that $R_M\subset(\Sigma_j)_M\times(\Sigma_l)_M$ has not FIP, and so $R_M\subset S_M$ has not FIP, a contradiction. Then, $(B_1)$ holds.

If there exists $M\in\Sigma\mathrm{Max}(R)$ such that $R_M$ is a $\Sigma$PIR, then $R$ is a $\Sigma$FMIR and $n=2$ by Theorem~\ref{4.2}. Moreover, since $R_M$ is not a field, Theorem~\ref{2.17} gives that for each $j\in\{1,2\}$, either $R_M\subset(\Sigma_j)_M$ is a special minimal ramified extension or $R_M=(\Sigma_j)_M$. Then $(B_2)$ holds.

Conversely, assume that $R$ is an FMIR and that $(B_1)$ and $(B_2)$ hold. Clearly, $\mathrm{MSupp}(S/R)$ is finite. Then, $R\subseteq S$ has FIP if and only if $R_M\subseteq S_M$ has FIP for each $M\in\mathrm{MSupp}(S/R)$ by Proposition~\ref{2.7}.

The integral closure of $R$ in $S$ is $\overline S=\prod_{j=1}^n\overline S_j$ by Proposition~\ref{2.4}  and $\overline S\subseteq S$ has FIP. Hence, $\overline S_M\subseteq S_M $ has FIP for each $M\in\mathrm {MSupp}(S/R)$. Then, $R\subseteq S$ has FIP if and only if the module finite extension $R_M\subseteq{\overline S_M}$ has FIP for each $M\in\mathrm{MSupp}(S/R)$  \cite[Theorem 3.13]{DPP2}.

If $R_M$ is finite, so is ${\overline S_M}$ and $R_M\subseteq{\overline S_M}$ has FIP. Now if $R_M$ is an infinite field, $R_M\subseteq\substack{+\\ \overline{S}_M}R_M$ as well as $R_M\subseteq{\overline S _M}$ have FIP. To see this, mimic the proof of Theorem~\ref{2.17}, using the fact that there is at most one $j\in\{1,\ldots,n\}$ such that $R_M\neq(\Sigma_j)_M$, so that $R_M=(\Sigma_l)_M$ for each $l\in\{1, \ldots,n\},\ l\neq j$. As in the proof of Theorem~\ref{2.17}, we get that $R_M\subset\substack{+\\ S_M}R_ M$ has FIP, because $\substack{+\\ S_M}R_M=R_M[\alpha]$, where $\alpha$ is the $n$-uple whose all components are $0$, except the $j$th which is $\alpha_j$ defining $(\Sigma_j)_M=R_M[\alpha_j]$. Lastly, if $R_M$ is a $\Sigma$PIR, then $n=2$ and Theorem~\ref{2.17} gives that $R_M\subseteq {\overline S_M}$ has FIP.

To conclude, $R\subseteq {\overline S}$ has FIP. 
\end{proof}

We can rephrase Theorem~\ref{4.2} in the following way.

\begin{corollary}\label{4.4} Let $R$ be a ring and $n>1$ an integer. Then, $R\subseteq R^n$ has FIP if and only if $R$ is Artinian and setting $\{M_1,\ldots,M_m\}:=\mathrm{Max}(R)$ and $\alpha_i := n(R_{M_i})$, then  for each $i$, one of the following conditions holds:

\begin{enumerate}
\item $\alpha_i=1$.

\item $|R/M_i|<\infty$. 

\item $R_{M_ i}$ is a SPIR and $n\leq 2$ as soon as there exists some $i$ such that $\alpha_i>1$ and $R_{M_ i}$ is a $\Sigma$PIR. 

\end{enumerate}

\end{corollary}

\begin{proof} By Theorem~\ref{4.2}, $R\subseteq R^n$ has FIP if and only if $R$ is a finite direct product $\prod_{i=1}^mR_{M_i}$ of finite local rings, SPIRs, and fields, with $n=2$ as soon as there is some $R_{M_i}$ which is a $\Sigma$PIR. Note that $0=\prod_{i=1}^mM_i^{\alpha_i}$ and set $R_i := R_{M_i}$ so that $0=M_i^{\alpha_i}R_i$.

Assume that $R\subseteq R^n$ has FIP and  fix some $i$. Then $R_i$ is a field if and only if $\alpha_i=1$, giving (1). We know that $R_{M_i}$ is a finite ring if and only if $|R/M_i|<\infty$, which gives (2). Assume that $\alpha_i>1$ and $|R/M_i|=\infty$. Then, $R_i$ is a $\Sigma$PIR, so that $n=2$ and we have (3). 

Conversely, assume that $R$ is an Artinian ring and  that for each $i$ one of conditions (1), (2) or (3) holds. It follows that $R$ is a finite direct product $\prod_{i= 1}^mR_i$ of primary rings.  We have just seen that $R_i$ is a field when $\alpha_i=1$. If $|R/M_i|=|R_i/M_iR_i|<\infty$, then $R_i$ is a finite ring. At last, if $\alpha_i>1$ and $|R/M_i|=\infty$, then $R_ {M_i}$ is a $\Sigma$PIR and $n\leq 2$. Now, use Theorem~\ref{4.2} to get that $R\subseteq R^n$ has FIP.
\end{proof}

Extensions of the form $R^p\subseteq R^n$, for some integers $1<p<n$ generalize extensions $R\subseteq R^n$. For $R^p$ and $R^n$ endowed with their canonical structures of $R$-algebras, we show that $\mathrm{Homal}_{R}(R^p,R^n)$ has at least $S(n,p)$ elements (the {\it Stirling number of the second kind} $S(n,p):=|P(n,p)|$ where $P(n,p)$ is the set of  partitions of $\{1,\ldots,n\}$ into $p$  subsets). We set $\mathrm{Exal}_R(R^p,R^n):=\{\varphi\in\mathrm{Homal}_{R}(R^p,R^n)\mid\varphi \ \mathrm{injective}\}$.

\begin{proposition}\label{4.6} Let $R$ be a ring and $1< p <n$ two  integers,  then:

\begin{enumerate}
\item $|\mathrm{Exal}_R(R^p,R^n)| \geq S(n,p)$.

\item If $R$ is connected, $|\mathrm{Exal}_R(R^p,R^n)| = S(n,p)$. 

\item If  $R\subseteq\mathrm{Tot}(R)$ is  t-closed and $\mathrm{Tot}(R)$ Artinian (for instance, if $R$ is Artinian), then $|\mathrm{Exal}_{R}(R^p,R^n)|\leq S(n,p)^{|\mathrm{Min}(R)|}$.
\end{enumerate}  
\end{proposition}

\begin{proof} 

Let $\mathcal{C}:=\{f_1,\ldots,f_p\}$  and $\mathcal{B}:=\{e_1,\ldots,e_n\}$ be the canonical bases of the $R$-algebras $R^p$ and $R^n$, that are complete families of orthogonal idempotents.

For  $\varphi \in \mathrm{Homal}_R(R^p,R^n)$, let $\lambda(\varphi) =(a_{i,j}) \in M_{n,p}(R)$ be its matrix  in the bases $\mathcal{C}$  and $\mathcal{B}$ (with the rule $\varphi(f_j)=\sum_{i=1}^na_{i,j}\cdot e_i$ for each $j$). Then $\lambda$ defines an injective map whose image $\Lambda$ we compute.   Applying  the ring morphism $\varphi$ to the relations $f_j^2=f_j$,  $f_jf_k=0$ for each $ j\neq k$ and $\sum_{j=1}^pf_j=1_{R^p}$, we   get  the conditions  $(*_1)$: $a_{i,j}^2=a_{i,j}$, $(*_2)$: $a_{i,j}a_{i,k}=0$ for each $ j\neq k$ and  $(*_3)$: $\sum_{j=1}^pa_{i,j}=1$, for each $i$.
It is easily seen that $\Lambda=\{(a_{i,j})\in M_{n,p}(R)\mid(*_1),(*_2),(*_3) \}$ and that $\lambda: \mathrm{Homal}_R(R^p,R^n) \to \Lambda$ is bijective. Indeed, any element of $\Lambda$ is the matrix of a ring morphism by $(*_1), (*_2), (*_3)$. 
 
(1) Let $H:=\{\varphi\in\mathrm{Exal}_R(R^p,R^n)\mid\lambda(\varphi)\in M_{n,p}(\{0,1\})\}$. For $\varphi\in H$ and $\lambda(\varphi)=(a_{i,j})$, we have $\ a_{i,j}\in\{1,0\}$ for each $(i,j)$ and then $a_{i,k}=0$ as soon as $a_{i,j}=1$ for some $j\neq k$ by $(*_2)$. For each $j\in\{1,\ldots,p\}$, set $A_j:=\{i\in\{1,\ldots,n\} \mid a_{i,j}=1\}$. Since $\varphi$ is injective, $\varphi(f_j)\neq 0$ for all $j$ implies that each $A_j\neq \emptyset$.
Then $(*_2)$ implies $A_j\cap A_k=\emptyset$ for $ j\neq k$ and $(*_3)$ that $\{1,\ldots,n\}=\cup_{j= 1}^pA_j$, since each $i\in\{1,\ldots,n\}$ is in one (and only one) $A_j$, so that $\{A_1,\ldots,A_p\}\in P (n,p)$. Hence, there is a map $\mu:H\to P(n,p)$, where $\mu(\varphi) = \{A_1,\ldots,A_p\}$, such that  $\varphi(f_j)=\sum_{i\in A_j}e_i$  for each $j$.
Then  $\mu$ is bijective because any element $\{A_1,\ldots,A_p\}$ of $P(n,p)$ defines some $\varphi \in H$ by the relations $\varphi(f_j)=\sum_{i\in A_j}e_i$  for each $j$.

(2) If $R$ is connected, $(*_1)$ implies  that $H = \mathrm{Exal}_R(R^p,R^n)$. 

(3)  If  $T:= \mathrm{Tot}(R)$ is  Artinian, then $T\cong\prod_{l=1}^mR_{M_l}$, where $\mathrm{Min}(R):=\{M_1,\ldots,M_m\}$. Since $R\subseteq T$ is  t-closed, the idempotents of $R$ and $T$ coincide. Then it is enough to use (2).
\end{proof}

We show that anything is possible when $R$ is   a $\Sigma$PIR.

\begin{proposition}\label{4.7}  Let $(R,M)$ be a $\Sigma$PIR, $p,n$ be two integers such that $1<p<n$ and $\varphi \in \mathrm{Exal}_R(R^p,R^n)$. The following statements hold:
\begin{enumerate}
\item If $n=p+1$,  $\varphi $ has FIP.

\item If $p+2\leq n\leq 2p$,   $\varphi$ has FIP in some cases and not FIP in some others. 

\item If $n\geq 2p+1$, then  $\varphi $ has not FIP. 
\end{enumerate}
\end{proposition}

\begin{proof} We keep notation of Proposition~\ref{4.6}(2). Since $R$ is connected, any extension  $\varphi$ of $R$-algebra  $R^p \subseteq R^n$ comes from some partition $\cup_{j=1}^pA_j$ of $\{1,\ldots,n\}$ with $\varphi(f_j)=\sum_{i\in A_j}e_i$.  In view of \cite[Lemma III.3]{DMPP}, we may identify $S:= R^n$ with $\prod_{j=1}^pS_j$, where $S_j:=\varphi(f_j)S$ is a ring extension of $R$ for each $j$. Moreover, $R^p\subseteq R^n$ has FIP if and only if each $ R\subseteq S_j$ has FIP   \cite[Proposition III.4]{DMPP}. But $S_j$ is the $R$-algebra generated by $\{e_i\mid i\in A_j\}$, and then isomorphic to $R^{|A_j|}$. Consider the following cases and use Theorem~\ref{4.2} for each $R\subseteq S_j$. 

\noindent (1) $n=p+1$. Then, $|A_j|=1$ for all $j$, except one $j_0$ such that $|A_{j_0}|=2$. It follows that  $S_j$ is isomorphic either to $R$, or $R^2$. In both cases, $R\subseteq S_j$ has FIP and $R^p\subseteq R^n$ has FIP.

\noindent (2) $p+2\leq n\leq 2p$.  We consider two subcases:

(a) If $|A_j|=1$ for all $j$, except one $j_0$ such that $|A_{j_0}|=n-p+1\geq 3$, then $R\subset S_{j_0}$ has not FIP, whence also $R^p\subseteq R^n$. 

(b) Set  $k:=n-p\leq p$ and consider a partition $\{A_1,\ldots,A_p\}$   such that $|A_j|=2$ for $j\leq k$ and $|A_j|=1$ for $j>k$.  Then, $R\subseteq S_j$ has FIP for each $j$ and so has  $R^p\subseteq R^n$. We have proved that $R^p\subseteq R^n$ has FIP or not according to the structure of $R^p$-algebra considered for $R^n$. 

\noindent (3) $n\geq 2p+1$. Consider a partition as above. If  $|A_j|\leq 2$ for all $j$, then $n\leq 2p$ is a contradiction. Hence, there is $j_0$ such that $|A_{j_0}|>2$. It follows that $R\subset S_{j_0}$ has not FIP and $R^p\subset R^n$ has not FIP. 
\end{proof}

\begin{proposition}\label{4.8} Let $R$ be a (resp$.$ connected) ring and $1< p <n$  two integers. Then, $\varphi \in \mathrm{Exal}_R(R^p,R^n)$ has FIP if (resp$.$ and only if) $R$ is an FMIR and $n\leq 2p$ when $R$ is a $\Sigma$FMIR. 
\end{proposition}

\begin{proof} We use the notation of the proof of Proposition~\ref{4.7} which holds for an arbitrary ring. Then, $R^p\subseteq R^n$ has FIP if and only if $R\subseteq S_j$ has FIP for each $j$. Fix a partition $\{A_1,\ldots,A_p\}$ of $\{1,\ldots,n\}$, so that $S_j\cong R^{|A_j|}$. Set $k_j=|A_j|$ and $k:=\sup\{k_j\}_{j= 1,\ldots,p}$. It follows that $R\subseteq S_j$ has FIP for each $j$ if and only if $R\subseteq R^k$ has FIP, since there are extensions $R^{k_j}\subseteq R^k$. But Theorem~\ref{4.2} shows that $R\subseteq R^k $ has FIP if and only if $R$ is a FMIR and $k\leq 2$ when $R$ is $\Sigma$FMIR. Assume that $R$ is a $\Sigma$FMIR. An easy calculation using the discussion of the proof of Proposition~\ref{4.7} leads to a partition $\{A_1,\ldots,A_p\}$ of $\{1,\ldots n\}$ such that $|A_j|\leq 2$ for each $j$ if and only if $n\leq 2p$, giving the wanted result. 

If $R$ is  connected,  Proposition~\ref{4.6}  tells us that  $\mathrm{Exal}_R(R^p,R^n)$ is  in bijection with  the set  $P(n,p)$ of  partitions $\{A_1,\ldots,A_p\}$ of $\{1,\ldots,n\}$. Assume that $ \varphi :R^p\hookrightarrow R^n$ has FIP, so that $R\subseteq S_j$ has FIP for each $j\in\{1,\ldots,p\}$. The first part of the proof shows that this holds if and only if $R$ is a FMIR and $k\leq 2$ when $R$ is a $\Sigma$FMIR, whatever is its associated  partition. 
\end{proof}

\section{Idealizations which are FCP or FIP extensions}

Let $M$ be an $R$-module. In this section, we  consider the ring extension $R\subseteq R(+)M$, where $R(+)M$ is the idealization of $M$ in $R$.  

Recall that $R(+)M:=\{(r,m)\mid (r,m)\in R\times M\}$ is a commutative ring whose operations are defined as follows: 

$(r,m)+(s,n)=(r+s,m+n)$ \ \   and  \ \ \ $(r,m)(s,n)=(rs,rn+sm)$

Then  $(1,0)$ is the unit of $R(+)M$, and $R \subseteq R(+)M$ is  a ring morphism defining $R (+)M$ as an $R$-module, so that we can identify any $r\in R$ with $(r,0)$. The following lemma will be useful for all this section. 

\begin{lemma}\label{5.1} Let  $M$  be an $R$-module, then $R\subseteq R(+)M$ is a subintegral extension with conductor $(0:M)$.
\end{lemma}

\begin{proof} If $(r,m)\in R(+)M$, then $(r,m)^2=2r(r,m)-r^2(1,0)$ shows that $R(+)M$ is integral over $R$. Moreover, by \cite[Theorem 25.1(3)]{H}, $\mathrm{Spec}(R(+)M)=\{P(+)M\mid P\in\mathrm{Spec}(R)\}$  implies that $R \subseteq R (+)M$ is subintegral.

Set $S:=R(+)M$ and let $x\in(R:S)$. Then, we have $(x,0)(0,m)=(0,xm)\in R$ for any $m\in M$, so that $ x\in(0:M)$. Conversely, any $x\in(0: M)$ gives $x(r,m)=(xr,0)\in R$ for any $(r,m)\in R(+)M$, which implies $x\in(R:S)$. So, we get $(R:S)=(0:M)$. 
\end{proof}

\begin{proposition}\label{5.2} Let  $M$ be  an $R$-module, then $R\subseteq R(+)M$ has FCP  if and only if ${\mathrm L}_R(M)< \infty$  and, if and only if $R/(0:M)$ is  Artinian and $M$ is f.g$.$ over $R$. 
\end{proposition}

\begin{proof} Set $S:=R(+)M$. Since $R\subseteq S$ is  integral, $R\subseteq S$ has FCP if and only if ${\mathrm L}_R(S/R)< \infty$  by \cite[Theorem 4.2]{DPP2}. By the same reference, this condition is equivalent to $R/(0:M)\cong R/(R:S)$ is  Artinian  and $R \subseteq S$ is module finite. But $R \subseteq S$ is  module finite implies that $S/R\cong M$ is also  f.g$.$. The converse is obvious. 
\end{proof}

For a submodule $N$ of an $R$-module $M$, we denote by $\llbracket N,M\rrbracket$ the set of all submodules of $M$ containing $N$ and set $\llbracket M\rrbracket:=\llbracket 0,M\rrbracket$. Recall that $M$ is called {\it uniserial} if $\llbracket M \rrbracket$ is linearly ordered. 

\begin{proposition}\label{5.3} Let $M$ be  an $R$-module, then $R\subseteq R(+)M$ is a $\Delta_0$-extension because $[R,R(+)M] = \{ R(+)N \mid N \in \llbracket M \rrbracket\}$.
\end{proposition}

\begin{proof} Set $S:=R(+)M$ and let $T\in\llbracket R,S\rrbracket$. Let $p:S\to M$ be the projection defined by $p(r,m)=m$, for any $(r,m)\in S$ and set $N:=p(T)$. Obviously, $T\subseteq R\times N$. Let $(r,n)\in R\times N$. There exists $x\in R$ such that $(x,n)\in T$. But $(r,0)$ and $(x,0)\in R\subseteq T$. This implies $(0,n)\in T$, which gives $(r,n)\in T$, so that $T=R\times N$ as $R$-submodules of $S$, and $T=R(+)N$ as $R$-subalgebras of $S$. 
\end{proof}

We say that an $R$-module $M$  is  an FMS module if $M$ has finitely many $R$-submodules. An FMS $R$-module $M$ is Noetherian and Artinian and $R/(0:M)$ is a Noetherian and Artinian ring. We denote by $\nu_R(M)$  (or $\nu (M)$) the number of submodules of an  FMS $R$-module $M$. Hence, $\nu(R)$ is the number of ideals of an FMIR $R$.  
\begin{proposition}\label{5.4} Let  $M$  be an $R$-module, then $R\subseteq R(+)M$ has FIP if and only if $M$  is an FMS module.  In this case, $|[R,R(+)M]| =\nu(M)$. 
\end{proposition}

\begin{proof} Set $S:=R(+)M$. By Proposition~\ref{5.3}, $[R,R(+)M] = \{ R(+)N \mid N \in \llbracket M \rrbracket\}$. It follows that $R\subseteq S$ has FIP if and only  if $M$ is an FMS module. In this case, $|[R,R(+)M]|= \nu(M)$.
\end{proof}

We now intend to characterize FMS modules by using the previous proposition.

\begin{theorem}\label{5.5} An $R$-module $M$ over a quasi-local ring $(R,P)$ is an FMS module if and only if  conditions $\mathrm{(1)}$ and $\mathrm{(2)}$ hold with $C=(0:M)$:

\begin{enumerate}
\item $M$ is finitely generated, and cyclic when $|R/P|=\infty$.

\item $R/C$ is an FMIR.
\end{enumerate}

If $M$ is an FMS $R$-module, $(R,P)$ is local, $|R/P|=\infty$, and $M=Re$ for some $e\in M$, then $M$ is uniserial, $\llbracket M\rrbracket=\{P^je\ |\ j=0,\ldots,m\}$, with $m:=n(R/C)=\nu(R/C)-1$ and $|[R,R(+)M]|=m+1$.

Assume in addition that $P=(0:M)$. Then $R\subseteq R(+)M$ has FIP if and only if $M$ is simple, if and only if  $R\subseteq R(+)M$ is minimal ramified. 
\end{theorem}

\begin{proof} Note that $R$-submodules and $R/C$-submodules of $M$ coincide.
Assume that $M$ is an FMS module. Then Proposition~\ref{5.4} shows that $R\subseteq R(+)M$ has FIP, whence has FCP. We deduce from Proposition~\ref{5.2} that $M$ is f.g$.$ and $(R/C,P/C)$ is  local Artinian. To prove (2), we consider two cases. If $|R/P|<\infty$, then $|R/C|<\infty$ (see the remark before Lemma~\ref{2.13}),  so that $R/C$ is an FMIR. 

Assume now that $|R/P|=\infty$.  Denote by $Re_1,\ldots,Re_n$ , with $e_i\in M$,  the finitely many cyclic submodules of $M$. Then for any $m\in M$, there is some $i$ such that $Rm=Re_i$. Hence, $M=\cup_{i=1}^nRe_i$. If $n=1$, then $M$ is cyclic and uniserial since $\llbracket M\rrbracket=\{0,M\}$. Assume that $n>1$ and that $M$ is not uniserial. Then $M$ has two incomparable cyclic submodules, for instance $Re_1$ and $Re_2$, and we may assume that $Re_2\not\subseteq Re_i$ for any $i\in\{1, \ldots,n\}\setminus\{2\}$. Let $\mathcal{F}$ be a(n infinite) set of representative of the non-zero elements of $R/P$. Then, each $\alpha\in\mathcal{F}$ is a unit of $R$. For each $\alpha\in\mathcal{F}$, set $m_ {\alpha}:=e_1+\alpha e_2$. Obviously $m_{\alpha}\not\in Re_1\cup Re_2$, so that $m_{\alpha}\in Re_i$, for some $i\neq 1,2$. Let $\alpha,\beta\in\mathcal{F},\ \alpha\neq\beta$. We claim that $m_{\alpha}$ and $m_{\beta}$ are not in the same $Re_i$. Deny, then $m_{\alpha}-m_{\beta}=(\alpha-\beta)e_2\in Re_i$ and $\alpha-\beta$ is a unit implies $e_ 2\in Re_i$, a contradiction. By the pigeonhole principle, this is absurd and $M$ is uniserial and necessarily   cyclic.

Assume that $M$ is an FMS module with $M=Re$ for some $e\in M$; so that $C=(0:e)$. Set $R':=R/C,\ P':=P/C$ and $I_N:=(N:_Re)$ for $N\in\llbracket M\rrbracket$. Then, $I_N\in\llbracket C,R\rrbracket$ and is such that $N=I_Ne$. Conversely, $I\in\llbracket C,R\rrbracket$ is such that $I=I_{Ie}$ with $Ie\in\llbracket M\rrbracket$, since $C\subseteq I$. We define a bijective map $\psi:\llbracket C,R\rrbracket\to\llbracket M \rrbracket$ by $I\mapsto Ie$. It follows that $R'$ is an FMIR (either a field or a SPIR) and $\nu(M)=\nu(R/C)$.

If $R'$ is a SPIR, there is some $x\in P$, whose class $\bar x\in R'$ is such that $P'=R'\bar x$, $\bar x^m =0$ and $\bar x^{m-1}\neq 0$, for $m:=n(R')>1$. It follows that $\llbracket C,R\rrbracket=\{P^j+C|j\in\{0 ,\ldots,m\}\}$ and $\llbracket M\rrbracket=\{P^je|j\in\{0,\ldots,m\}\}$ (to see this, use the bijection $\psi$). If $R'$ is a field, then $P=C$ gives $m=1$.

Now, assume that (1) and (2) hold. There is no harm to suppose that $C=0$ and that $R$ is an FMIR, so that $(R, P)$ is local Artinian. If $|R/P|<\infty$, we get that $|M|<\infty$ and then $M$ is an FMS module. Assume that $|R/P|=\infty$, and that $M=Re$ is cyclic. If $P=0$, then $M$ is a one-dimensional vector space over the field $R$, so that $\nu(M)=2=\nu(R)$. If $P\neq 0$, consider $S:=R(+)M=R+Rf$, where $f=(0,e)$. From \cite[Proposition 4.12]{Gil} we deduce that $|\llbracket R,S\rrbracket|<\infty$, since $R$ is an FMIR and also that there is a bijective map $\llbracket R,S\rrbracket\to\llbracket M\rrbracket$. In fact $\llbracket R,S\rrbracket=\{R(+)N \mid N\in\llbracket M\rrbracket \}$. By Proposition~\ref{5.3},  $M$ is an FMS module.

To end, assume that $(R,P)$ is quasi-local with $|R/P|=\infty$. Let $M$ be a simple $R$-module, with $P=(0:M)$. Then $[R,R(+)M]=\{R,R(+)M\}$ by Proposition~\ref{5.3}. It follows that $R\subseteq R(+)M$ has FIP and is a minimal ramified extension since minimal subintegral.
\end{proof}

\begin{corollary}\label{5.6} Let  $M$ be an $R$-module and $C:=(0:M)$. Then $M$ is an FMS module if and only if the two following conditions hold:

\begin{enumerate}
\item $M$ is f.g$.$ and $M_P$  cyclic  for all $P\in{\mathrm V}(C)$ such that $|R/P|=\infty$.

\item $R/C$ is an FMIR.
\end{enumerate}

In case $\mathrm{(1),(2)}$ both hold, set $\{P_1,\ldots,P_n\} ={\mathrm V}(C)$ and suppose that each $|R/P_i|=\infty$. Then $M$ is generated by some $e_1,\ldots,e_n\in M$, such that $M_{P_i}=R_{P_i}(e_i/1)$ for each $i$. 
\end{corollary}

\begin{proof} If $M$ is an FMS module, Proposition~\ref{5.4} shows that $R\subseteq R(+)M$ has FIP, and then has FCP. Then, $M$ is f.g$.$ and $R/C$ is  Artinian by Proposition~\ref{5.2}. Let $P\in{\mathrm V}(C)$, then $M_P$ is an FMS module, so that we can use Theorem~\ref{5.5}. Then $R_P/C_P\cong(R/C)_P$ is an FMIR, and so is $R/C$, since $|{\mathrm V}(C)|<\infty$, which gives (2). Moreover, for  $P\in{\mathrm V}(C)$ with $|R/P|=\infty$, Theorem~\ref{5.5} gives that  $M_P$ is cyclic and (1) holds.

Conversely, if (1) and (2) hold, they also hold for each $M_P$, where $P\in{\mathrm V}(C)$. Theorem~\ref{5.5} gives that $M_P$ is an FMS module for any $P\in{\mathrm V}(C)$. To show that $M$ is an FMS module, there is no harm to suppose that $C= 0$, so that $R$ is Artinian, with $\mathrm{Max}(R)= \{P_1,\ldots,P_n\}$. Now if $N$ is a submodule of $M$, it is well known that $N = \cap_{i=1}^n \varphi_i^{-1}(N_{P_i})$, where $\varphi_i :M \to  M_{P_i}$ is the natural map and $M$ is an FMS module.

Now, assume that (1) and (2) hold and that $|R/P|=\infty$ for any $P\in{\mathrm V}(C)=\{P_1,\ldots,P_n\}$.  For each $j=1, \ldots, n$, there is some $e_j\in M$ such that $M_{P_j} = R_{P_j}(e_j/1)$. Set $M' := Re_1+\cdots+Re_n$. It is easy to show that  $M'_{P_j}=M_{P_j}$ for $j=1\dots,n$.  Observe that ${\mathrm V}(C)=\mathrm{Supp}(M)$, because $M$ is f.g$.$ (\cite[Proposition 17, ch. II,  p.133]{Bki AC}). Now let $P\in\mathrm{Max}(R)\setminus{\mathrm V}(C)$. We get that $M'_P\subseteq M_P=0$ and then $M'=M$.
\end{proof}

Let  $N$  be a submodule of an $R$-module. By Proposition~\ref{5.3}, $R(+)N$ is an $R$-subalgebra of $R(+)M$ and then $R(+)M$ is an ($R(+)N$)-algebra. Even if $R\subseteq R(+)M$ has not FCP (resp$.$ FIP), it may be that $R(+)N\subseteq R(+)M$ has FCP (resp$.$ FIP).  

Any ($R(+)N$)-subalgebra of $R(+)M$ is an $R$-subalgebra of $R(+)M$, and then is of the form $R(+)N' $, for some $N'\in\llbracket N,M\rrbracket$ since $R(+)N\subseteq R(+)N'$. Conversely, for any $R$-subalgebra $N'$ of $M$ containing $N$, $R(+)N'$ is an ($R(+)N$)-subalgebra of $R(+)M$. In particular, $R(+)N\subseteq R(+)M$ is a minimal extension if and only if $M/N$ is a simple module. 

\begin{proposition}\label{5.7} Let  $N$ be a submodule  of an $R$-module $M$. Then:

\begin{enumerate}
\item  $R(+)N\subseteq R(+)M$ is a $\Delta_0$-extension.

\item $R(+)N\subseteq R(+)M$ has FCP if and only if ${\mathrm L}_R(M/N)<\infty$. In this case, $\ell[R(+)N,R(+)M]={\mathrm L}_R(M/N)$.

\item $R(+)N\subseteq R(+)M$ has FIP if and only if $M/N$ is an FMS module. In this case, $|[R(+)N,R(+)M]|= \nu(M/N)$.
\end{enumerate}
\end{proposition}

\begin{proof} (1) By Proposition~\ref{5.3}, $R\subseteq R(+)M$ is a $\Delta_0$-extension. Since an ($R (+)N$)-submodule $S$ of $R(+)M$ containing $R$ is also an $R$-submodule of $R(+)M$, we get that $S$ is a ring, so that $R(+)N\subseteq R(+)M$ is a $\Delta_0$-extension.

(2) By Lemma~\ref{5.1}, $R\subseteq R(+)M$ is  integral and so is $R(+)N\subseteq R(+)M$. Therefore, the following conditions are equivalent:

- $R(+)N\subseteq R(+)M$ has FC

- there exists a finite chain of minimal finite extensions going from $R(+)N$ to $R(+)M$ (\cite[Theorem 4.2(2)]{DPP2}) 

- there is a finite maximal chain of $R$-submodules of $M$ going from $N$ to $M$ 

-
 ${\mathrm L}_R(M/N)<\infty$. 
 
 In this case, $\ell[R(+)N,R(+)M]={\mathrm L}_R(M/N)$,  the supremum of the lengths of chains of submodules of $M$ containing $N$.

(3) The following conditions are equivalent:

- $R(+)N\subseteq R(+)M$ has FIP 

- there are finitely many $R(+)N$-subalgebras of $R (+)M$ 

- there are finitely many $R$-subalgebras of $R(+)M$ containing $R(+)N$ 

- there are finitely many $R$-submodules of $M$ containing $N$ 

- $M/N$ is an FMS module. 

In this case, $|[R(+)N,R(+)M]|$  is also the number of $R$-subalgebras of $M$ containing $N$, which is also $\nu(M/N)$.
\end{proof}

We consider now the special case where $M$ is  an ideal $I$ of $R$.

\begin{proposition}\label{5.8} Let $I$ be an ideal of a ring $R$, $S:= R(+)R$ and $T:= R(+)I$. Then:

\begin{enumerate}
\item $R\subseteq S$ has FCP if and only if ${\mathrm L}_R(R)<\infty$ if and only if $R$ is Artinian. In this case, $\ell[R,R(+)R]={\mathrm L}_R(R)$.

\item $R\subseteq T$ has FCP if and only if ${\mathrm L}_R(I)<\infty$ if and only if $I$ is finitely generated and $R/(0:I)$ is  Artinian. In this case, $\ell[R,R(+)I]={\mathrm L}_R(I)$.

\item $R\subseteq S$ has FIP if and only if $R$ is an FMIR. In this case, $|[R,R(+)R]|=\nu(R)$.

\item $R\subseteq T$ has FIP if and only if $\llbracket  I \rrbracket$ is finite. In this case, $|[R,R(+)I]|= \nu(I)$.
\end{enumerate}
\end{proposition}

\begin{proof} Proposition~\ref{5.7} gives most of results, taking $N=0$, because $R(+)0\cong R$ with $M$ equal either to $R$ or $I$. 

(1) By \cite[Theorem 7, p.24]{N}, ${\mathrm L}_R(R)<\infty$ if and only if $R$ is  Artinian.

(2) If ${\mathrm L}_R(I)<\infty$, then \cite[Proposition 9, p.22 and Theorem 7, p.24]{N} give that $I$ is of finite type and it follows that $R/(0:I)$ is Artinian  by \cite[Corollary of Theorem 2, p.181]{N}. Conversely, this corollary gives that ${\mathrm L}_R(I)<\infty$ when $I$ is finitely generated and $R/(0:I)$ is  Artinian.
\end{proof}

\begin{proposition}\label{5.9}  Any f.g$.$ module over a ring $R$ is an  FMS module if and only if $R$ is a finite ring.
\end{proposition}

\begin{proof} If $R$ is  finite, then  $\llbracket M\rrbracket $ is finite for any f.g$.$ $R$-module $M$. 

Conversely, let $R$ be a ring such that any f.g$.$ $R$-module is an FMS module. Set $S:=R [X,Y]/(X^2,XY,Y^2)=R[x,y]$, where $x$ and $y$ are respectively the classes of $X$ and $Y$ in $S$. Then $S$ is an $R$-module with basis $\{1,x,y\}$. For each $\alpha\in R$, set $S_{\alpha}:=R(x+\alpha y)$, which is an $R$-submodule of $S$. If $\alpha,\beta\in R,\  \alpha\neq\beta$, then $S_{\alpha}\neq S_{\beta}$. Therefore,  $|R|=\infty$ gives a contradiction and  $R$ is a finite ring.
\end{proof}

\begin{remark}\label{5.10} If $N$ is a submodule of an $R$-module $M$, Proposition~\ref{5.2} shows that $R\subseteq R(+)M$ has FCP if and only if $R\subseteq R(+)N$ and $R\subseteq R(+)(M/N)$ have FCP. This property does not hold for FIP. It is enough to consider a 2-dimensional vector space $M$ over an infinite field, and a 1-dimensional subspace $N$ because $N$ and $M/N$ are FMS modules, while $M$ is not.
\end{remark}

\begin{example}\label{5.11} In the following examples, we mix properties of the former sections. 

(1) Let $k$ be a field, $n> 1$ an integer, $E$ an $n$-dimensional $k$-vector space with basis $\{e_1, \ldots,e_n\}$ and set $R:=k^n$. We can equip $E$ with the structure of an $R$-module by the following law: for $(a_1,\ldots,a_n)\in R$ and $x=\sum_{i=1}^nx_ie_i,\ x_i\in k$, we set $(a_1,\ldots,a_n)x:=\sum_{i =1}^na_ix_ie_i$. Then $E$ is generated over $R$ by $\{e_1,\ldots,e_n\}$ and faithful, while $R$ is an FMIR. Finally, the prime (maximal) ideals of $R$ are the ideals $P_i:=\{(a_1,\ldots,a_n)\in R\ |\ a_i=0\}$ for $i=1,\ldots,n$, so that $R_{P_i}\cong k$. The canonical base $\{\varepsilon_1,\ldots,\varepsilon_n\}$  of $R$ over $k$ is such that each $\varepsilon_i\notin P_i$. We have $\varepsilon_ie_j=0$ for each $i,j\in \{1,\ldots,n\}$ such that $i\neq j$, so that $e_j/1=0$ in $R_{P_i}$ for $j\neq i$. It follows that $E_{P_i}=\sum_{j=1}^nR_{P_i}(e_j/1)=R_{P_i}(e_i/1)$ is cyclic over $R_{P_i}\cong k$. Then, whatever $|k|$ may be, Corollary~\ref{5.6} gives that $E$ is an FMS $R$-module. But, as soon as $|k|=\infty$ and $n\geq 2$, $|\llbracket E\rrbracket |$ is infinite (as a $k$-module). Since $E_{P_i}\cong k(e_i/1)$ is one-dimensional over $k$, $E_{P_i}$ has only two $R_{P_i}$-submodules). Set $F:=\prod_{i=1}^nE_{P_i}$ and consider the canonical injective morphism of $R$-modules $\varphi:E\to F$ and the projections $\varphi_i:F\to E_ {P_i}$. Any $R$-submodule $N$ of $F$ is of the form $N':=\prod_{i=1}^nN_i$, where $N_i=\varphi_i(N)$, because $N\subseteq N'\subseteq\sum_{i=1}^n\varepsilon_i N$. Now $\varphi$ is a $k$-isomorphism  because $\mathrm{Dim}_k(E)=\mathrm{Dim}_k(F)$, whence an $R$-isomorphism. It follows that $\nu_R (E)=2^n$.

 By Proposition~\ref{5.4},  $k^n\subseteq k^n(+)E$ has FIP, and $k\subseteq k^n$ has FIP by Proposition~\ref{2.1}. 
But, always in view of Proposition~\ref{5.4}, if $|k|=\infty$ and $n\geq 2$, then $k\subseteq k(+)E$ has not FIP, so that $k\subseteq k^n(+)E$ has not FIP.

(1') We keep the context of (1). Set $\mathcal R:=\prod_{i=1}^n(k/(0:e_i))$. Since $(0:e_i)=0$ for each $i$, we get $\mathcal R=k^n$. Then $k\subset k^n$ has FIP while $k\subseteq k(+)E$ has not FIP.

(2) Let $k$ be an infinite field, $n>1$ an integer and $E$ an $n$-dimensional vector space over $k$. Let $u\in\mathrm{End}(E)$ with minimal polynomial $X^n$. Then, $u^n=0$ and $u^{n-1}(e_1)\neq 0$ for  some $e_1\in E$. If $e_i:=u^{i-1}(e_1)$ for any $i\in\{1,\ldots,n\}$, an easy induction shows that $\{e_1, \ldots,e_n\}$ is a basis of $E$ over $k$. Set $R:=k[u]$, then $E$ is a faithful $R$-module with scalar multiplication defined by $P(u)\cdot x:=P(u)(x)$, for $P(X)\in k[X]$ and $x\in E$. Since $R\cong k[X]/(X^n)$ is a SPIR and $E=R\cdot e_1$ because $e_i=u^{i-1}\cdot e_1$ for each $i$, then by Theorem~\ref{5.5}, $E$ is an FMS $R$-module and $R\subseteq R(+)E$ has FIP by Proposition~\ref{5.4}. 

(2') We keep the context of (2). Since $u^n=0,\ u^{n-1}(e_1)\neq 0$ and $e_j=u^{j-1}(e_1)$ for any $j\in\{1,\ldots,n\}$, a short calculation gives $I_j:=(0:_Re_j)=Ru^{n-j+1}$. Then, $\cap_{j=1}^nI_j=0$ because $I_1=Ru^n=0$ and $\{I_1,\ldots,I_n\}$ is a separating family such that $I_j\subset I_{j+1}$ for each $j\in
\{1,\ldots,n-1\}$. Moreover, $R/I_j=R/Ru^{n-j+1}\cong k[X]/(X^{n-j+1})$. Set $M:=Ru,\ \mathcal R:=\prod_ {i=1}^n(R/(0:e_i))$ and $J_j:=(\cap_{k=1,k\neq j}^nI_k)$. Then, $J_1=I_2\cong(X^{n-1})/(X^n)$ and $J_j= 0$ for each $j>1$. Apply Corollary~\ref{3.8}. We have $\sum_{j=1}^nJ_j=I_2$, giving that $R/\sum_{j=1}^ nJ_ j=R/I_2\cong k[X]/(X^{n-1})$ is a SPIR and $|R/M|=\infty$, because $R/M\cong k$. Since $I_1+J_1=I _2\cong(X^{n-1})/(X^n)$ and $I_j+J_j=I_j\cong(X^{n-j+1})/(X^n)$ for each $j>1$, it is enough to take $n>3$ to get that $R\subset \mathcal R$ has not FIP.

(3) Let  $M=\sum_{i=1}^nRe_i$ be a faithful Artinian $R$-module and set $\mathcal R:=\prod_{i=1}^n(R/(0:e_i))$.  Then, $R$ is an Artinian ring in view of \cite[Theorem 2, page 180]{N}. Since $(0:M)=\cap_{i=1}^n(0:e_i)=0$, the family $\{(0:e_i)\}_{i=1,\ldots,n}$  is separating and $R\subseteq \mathcal{R}$ has FCP by Proposition~\ref{3.1}.

Examples (1') and (2') show that for a finitely generated $R$-module $M=\sum_{i=1}^nRe_i$ such that $\{(0:e_1),\ldots,(0:e_n)\}$ is a separating family, we may have only one of the two extensions $R\subseteq R(+)M$ and $R\subseteq \prod_{i=1}^n(R/(0:e_i))$ which has FIP, and not the other one. 

(4) Let $k$ be an infinite field, $n>1$ an integer and $E$ an $n$-dimensional vector space over $k$. Let $u\in\mathrm{End}(E)$ with minimal polynomial $\pi_u(X):=\prod_{i=1}^sP_i^{\alpha_i}(X)$, with each $P _i(X)\in k[X]$  of degree $1$, $P_i(X)\neq P_j(X)$ for $i\neq j$, and such that $n=\sum_{i=1}^s\alpha_i$. For each $i$, set $E_i:=\ker(P_i^{\alpha_i}(u))$. The Kernel Lemma gives that $E=\bigoplus_{i=1}^sE_i\ (*)$, with $\alpha_i=\dim_k(E_i)$. If $R:=k[u]$, then $E$ is a faithful $R$-module for the scalar multiplication defined by $P(u)\cdot x:=P(u)(x)$, for $P(X)\in k[X]$ and $x\in E$. Since $R\cong k[X]/\pi_u (X)$ is an Artinian FMIR, to conclude that $E$ is an FMS module over $R$ by applying Corollary~\ref{5.6}, we need only to show that $E_M$ is cyclic for each $M\in\mathrm{Max}(R)=\{M_1,\ldots,M_s\}$ where $M_i:=P_i(u)R$. We next prove that $E_{M_i}\cong(E_i)_{M_i}$ as $R_{M_i}$-modules. Let $x\in E_j$ for some $j\neq i$, then $P_j^{\alpha_j}(u)(x)=0$ and $P_j^{\alpha_j}(u)$ is a unit in $R_{M_i}$ since $P_j(X)\not\in(P_i(X))$. It follows that $x/1=0$ in $E_{M_i}$, so that $E_{M_i}\cong(E_i)_{M_i}$ by $(*)$. Now, we are reduced to (2) with $P_i^{\alpha_i}(u)=0$ in $(E_i)_{M_i}$, so that each $(E_i)_{M_i}$ is cyclic over $R_{M_i}$  and Corollary~\ref{5.6} holds.
\end{example}

\begin{theorem}\label{5.12}  A faithful $R$-module $M$ is an FMS module if and only if the two following conditions are satisfied:

\begin{enumerate}
\item $R$ is an FMIR which is a direct product of two rings $R'\times R''$, where $|R'|<\infty$ and $|R''/P|=\infty$ for any $P\in\mathrm{Spec}(R'')$.

\item $M$ is the direct product of a finite $R'$-module and a rank one projective $R''$-module.
\end{enumerate}
\end{theorem}

\begin{proof} If $M$ is an FMS module, $R$ is an FMIR and $M$ is f.g$.$ over $R$ by Corollary~\ref{5.6}. Then by Proposition~\ref{2.3}, $R=\prod_{i=1}^nR_i$, a product of local rings that are either  finite, or a SPIR, or a field. Let $R'$ be the ring product of the $R_i$ that are finite and $R''$ the product of the others. Then $|R'|<\infty$ and a SPIR factor $(R_i,P_i)$ of $R''$ is such that $|R_i/P_i|=\infty$ because $R_i$ is local Artinian. When $R_i$ is an infinite field, take $P_i=0$. So, (1) holds with $R=R' \times R''$.

Set $M':=R'M=\{(r',0)m\mid r'\in R',\ m\in M\}$ and $M'':=R''M=\{(0,r'')m\mid r''\in R'',\ m\in M\}$. By \cite[Remarque 3, ch.II, p.32]{Bki A1}, we get $M=M'\bigoplus M''\cong M'\times M''$, $R'M''=R''M'=0$ and $(0:_{R''}M'')=0$. Clearly, $|M'|<\infty$ since $M'$ is f.g$.$ over the finite ring $R'$. In the same way, $M''$ is f.g$.$ over $R''$. Now an $R''$-submodule $N$ of $M''$ gives an $R$-submodule of $M$ by the one-to-one function $N \mapsto M'\times N$. It follows that $M''$ is an FMS $R''$-module. Therefore, we can assume that $R$ is an FMIR with $|R/P|=\infty$ for each $P\in\mathrm{Spec}(R)=\{P_1,\dots,P_n\}$. By Corollary~\ref{5.6}, $M$ is generated over $R$ by some $e_1,\ldots,e_n\in M$ such that $M_{P_i}=R_ {P_i}(e_i/1)$ for each $i$. Actually, $e_i/1$ is free over $R_{P_i}$: suppose that $(a/t)(e_i/1)=0$ for $a\in R$ and $t\in R\setminus P_i$. There is some $s_i\in R\setminus P_i$ such that $s_iae_i=0$. Moreover, $ e_j/1\in M_{P_i}$ for $j\neq i$ allows us to pick up some $s_j\in R\setminus P_i$ such that $s_jae_j=0$. Setting $s:=s_1\cdots s_n$, we get $sae_k=0$ for each $k\in\{1,\ldots,n\}$. Since $M$ is faithful, $sa=0$ and $a/t=0$. By \cite[Th\'eor\`eme 2, ch.II, p.141]{Bki AC}, $M$ is a rank one projective $R$-module and (2) follows.

Conversely, assume that (1) and (2) hold and keep the  above notation  with $R=R'\times R''$,  $|R'|<\infty$, $|R''/P|=\infty$ for any $P\in\mathrm{Spec}(R'')$ and $M=M'\times M''$, where $M'$ is a finite $R'$-module and $M''$ is a rank one projective $R''$-module. Then, from \cite[Th\'eor\`eme 2, ch. II, p. 141]{Bki AC}, we deduce  that $M''$ is f.g$.$ over $R''$, with $M''_{P}$ cyclic for each maximal ideal $P$ of $R''$. Since $M'$ is also f.g$.$ over $R'$ because finite, $M$ is f.g$.$ over $R$.  For each $N\in\mathrm{Max}(R)$ such that $|R/N|=\infty$, there exists $P\in\mathrm{Max}(R'')$ such that $N=R'\times P$ and in this case $M_N\cong M''_{P}$ as $R_N$-modules. Indeed, consider the $R_N$-linear isomorphism $u:M_N\cong(M'\times M'')_{R'\times P}\to M''_P$ defined by $u((m',m'')/(s,t))=m''/t$,  using the ring isomorphism $R_N\cong R''_P$. It follows that $M_{N}$ is cyclic over $R_N$. By Corollary~\ref{5.6}, we can conclude that $M$ is an FMS module.
\end{proof}

We end this section by two results about quadratic extensions. According to \cite{CDL1}, an extension $R\subset S$ is called {\it pointwise minimal} if $R\subset R[t]$ is minimal for each $t\in S\setminus R$. 

\begin{proposition}\label{5.13} Let $R\subset S$ be a quadratic seminormal infra-integral FIP extension, where $(R,M, k)$ is a quasi-local ring with $k:= R/M$. Then, $R\subset S$ is pointwise minimal, and  minimal  as soon as $|k|>2$.
\end{proposition}

\begin{proof} Since $R\subset S$ is quadratic, the $R$-module $S_t:=R+Rt$ is a ring for each $t\in S\setminus R$. Moreover, $C_t:=(R:S_t)$ is a radical ideal of $S_t$ and $R$ and $R/C_t$ is Artinian by  \cite[Lemma 4.8, Theorem 4.2]{DPP2}, so that $C_t=M$. From $\dim_{k}(S_t/M)\leq 2$, we deduce that $ k\subset S_t/M$ is minimal, and so is $R\subset S_t$. Hence $R\subset S$ is pointwise minimal. In fact, by \cite[Proposition 4.9]{DPP2}, $(R:S)=M=\cap_{i=1}^nM_i$, where $\{M_1,\ldots,M_n\}:=\mathrm{Max}(S)$ and $R/M\subseteq S/M\cong\prod_{i=1}^nS/M_i\cong(R/M)^n$ is also quadratic. We claim that if $k\subseteq k^n$ is quadratic and $|k|>2$, then $n=2$. Deny and let $\{e_1,\ldots,e_n\}$ be the canonical basis of $k^n$. Let $f:=e_1+\alpha e_2$ for $\alpha\in k\setminus\{0,1\}$. Then $f^2=e_1+\alpha^2e_2\in k+kf$, so that $a\sum_{i=1}^ne_i+b(e_1+\alpha e_2)=f^2$ for some $a,b\in k$. Since $n>2$, we get $a=0 $, so that $b=1$ and $\alpha^2=\alpha$, a contradiction. Then, $n=2$, $k\subset k^2$ is minimal, and so is $R\subseteq S$. However, if $k=\{0,1\}$, then $k\subseteq k^n$ is quadratic, seminormal, infra-integral and FIP (each element of $k^n$ is idempotent), even if $n>2$ (Proposition~\ref{2.1}). 
\end{proof}

\begin{remark}\label{5.14} Let $R\subset S$ be a minimal extension. Then $R\subset S$ is a quadratic extension if and only if $R\subset S$ is a $\Delta_0$-extension.
\end{remark}

\begin{proof} A minimal extension is obviously a $\Delta$-extension. Moreover, $R\subset S$ is a $\Delta_0$-extension if and only if $R\subset S$ is a quadratic $\Delta$-extension by \cite[Proposition 5]{HP}, giving the result.
\end{proof}

\section{Etale and separable morphisms}
A finite extension of fields has FIP when it is separable. We now examine separable (\'etale) extensions. For instance, let $R$ be an Artinian reduced ring and  $f_1,\ldots, f_n \in R$ such that $(f_1,\ldots,f_n) = R$; we set $S:= \prod_{i=1}^n R_{f_i}$. It is known that $R\to S$ is a faithfully flat \'etale morphism. For each $P\in\mathrm{Spec}(R)$, we get a morphism of the form $R_ P\to S_P=(R_P)^p$ where $p$ is the number of open subsets $\mathrm{D}(f_i)$ containing $P$ while  $R_P$ is a field. In view of \cite[Proposition 3.7]{DPP2}, $R \to S$ is an FIP extension by Proposition~\ref{2.1}. We generalize below this example.

 Before that, we introduce some terminology for ring morphisms $R\to S$ of finite type ($R$-algebras $S$ of finite type). For some authors, a separable morphism $R\to S$ is such that $S\otimes _RS\to S$ defines $S$ as a projective $S\otimes_RS$-module \cite{DMI}, a definition we keep in the rest of the paper. In case $R\to S$ is of finite type, then $R\to S$ is separable if and only if its $S$-module of $K\ddot{a}hler$ differential $\Omega_{S|R}=0$. In Algebraic Geometry, such a morphism is called either {\it formally neat} or {\it formally unramified}. It is well known that a morphism is \'etale if and only if it is flat of finite presentation and formally neat.
 
 \begin{proposition}\label{6.1} Let $f: R \subseteq  S$ be a separable ring extension such that $R$ is Artinian and reduced. The following statements hold:
\begin{enumerate} 
\item  $R \to S$ is (module-)finite and \'etale.
\item $R \to S$ has FIP and is seminormal.
\end{enumerate}
Moreover, $|[R,S]|\leq\prod_{M\in\mathrm{MSupp}(S/R)}B_{n_M} $, where $n_M=\mathrm{dim}_{R/M}(S/MS)$.
 \end{proposition}
\begin{proof} (1) We first note that $\mathrm{MSupp}(S/R)$ is finite. Thus it is enough to show that the separable morphism $R_M\to S_M$ is finite for $M\in\mathrm{Spec}(R)$. Since $R_M$ is a field, this follows from \cite[Corollary 2.2, p.48]{DMI}. To complete the proof, observe that $R$ is absolutely flat and Noetherian.
 
 (2) We can assume that $R$ is a field by using \cite[Proposition 3.7]{DPP2}. By \cite[Definition 1, Ch. 5, p.28]{Bki A}, there is a base change $R \to L$, where $L$ is a field extension of $R$, such that $L\otimes_R S \cong L^n$. Then $n$ is necessarily $ \mathrm{dim}_R(S)$.  The result follows from  Proposition 2.1, \cite[Th\'eor\`eme 2.29]{Pic 3} and \cite[Lemma 2.1]{DPP3}, since the base change $R\to L$ is faithfully flat.
 
 The last statement is a consequence of \cite[Theorem 3.6]{DPP2}.
\end{proof}

\begin{theorem}\label{6.2}  Let $R\subseteq S$ be a (module-)finite \'etale extension with conductor $C$. Then 

\begin{enumerate} 
\item  $R\subseteq S$ has FIP and is seminormal if and only if $R/C$ is Artinian and reduced.

\item $R\subseteq S$ has FIP and $R/C$ is reduced if and only if $R\subseteq S$ is seminormal and $R/C$ is Artinian.
\end{enumerate}
\end{theorem}
\begin{proof} (1) Assume that $R/C$ is Artinian and reduced. In that case $R/C\subseteq S/C$ is module-finite and separable. In view of Proposition~\ref{6.1}, this extension has FIP and is seminormal and so is $R\subseteq S$.

Conversely, if $R\subseteq S$ has FIP, it has FCP and then $R/C$ is Artinian by \cite[Theorem 4.2]{DPP2}. Moreover, the seminormality of $R\subseteq S$ entails that $C$ is semi-prime in $S$ \cite[Lemma 4.8]{DPP2}, whence in $R$.

(2) If $R\subseteq S$ has FIP, then $R/C$ is Artinian by (1), which gives that $R\subseteq S$ is seminormal when moreover, $R/C$ is reduced.

Always by (1), the seminormality of $R\subseteq S$ entails that $R/C$ is reduced, and then $R\subseteq S$ has FIP, whereas $R/C$ is Artinian.
\end{proof}

As a consequence of the above theorem, we see that a seminormal finite separable extension has FIP if and only if it has FCP. Moreover, it follows from Proposition~\ref{6.1} that a separable extension $R\subseteq S$ whose conductor is a maximal ideal of $R$ has FIP.

\begin{remark}\label{6.3} Let $A$ be a ring, $p(X)\in A[X]$ and $B:=A[X]/(p(X))$, with $p(X)$ monic, so that $f:A\to B$ is faithfully flat. It is easy to show, using \cite[Lemma 2.6]{GP2}, that $f$ is infra-integral if and only if $p(X)$ splits in each $\kappa(P)[X]$ for $P\in\mathrm{Spec}(A)$, so that each fiber morphism  $\kappa(P)[X]\to\kappa(P)[X]\otimes_AB$ is of the form $\kappa(P)\to\kappa(P)^n$ for some integer $n$. It follows that  $f$ is \'etale if $f$ is infra-integral. Since the conductor of $f$ is $0$, when $f$ is infra-integral,  $f$ has FIP and is seminormal if and only if $A$ is Artinian reduced.
\end{remark}

\begin{remark}\label{6.4} Let $k\subseteq K$ be a finite separable extension of fields with minimal polynomial $f(X)\in k[X]$ and $f(X):=(X-\alpha)f_1(X)\cdots f_r(X)$ its decomposition into irreducible factors in $K[X]$. There are ring morphisms $p_i:k[X]/(f(X))\cong K\to K[X]/(f_i(X))$ for $i=1,\dots,r$. Denoting by $L_i$ the pullback fields associated to the morphisms $p_i$ and $K\to K[X]/(f_i(X))$, we get subextensions $k\subseteq L_i$ of $ k\subseteq K$. Each subextension of $k\subseteq K$ is an intersection of some of the $L_i$s \cite[Theorem 1]{HKN}.
 
 This last result can be generalized to FCP extensions, by using Noetherian lattices. Let $R\subseteq S$ be an FCP extension, then $[R,S]$ endowed with the inclusion is a lattice for  intersection and compositum. An element $T\in[R,S]$ is called $\cap$-irreducible (resp$.$ comp-irreducible) if $T=T_1 \cap T_2$ (resp$.$ $T=T_1T_2$) implies $T=T_1$ or $T=T_2$. It is clear that $T\in[R,S]$ is $\cap$-irreducible (resp$.$ comp-irreducible) if and only if either $T=S$ (resp$.$ $T=R$) or there is a unique $T' \in[R,S]$ such that $T\subset T'$ (resp$.$ $T'\subset T$) is a minimal extension. Then by \cite[Proposition 1.4.4]{NO}, any $T\in [R,S]$ is a finite intersection (resp$.$ compositum) of $\cap$-irreducible (resp$.$ comp-irreducible) elements of $[R,S]$.

\end{remark}

\end{document}